\documentclass[12pt]{amsart}
\usepackage{xcolor}
\usepackage{graphicx}
\usepackage{amssymb}
\usepackage{amsfonts}
\usepackage{hyperref}
\usepackage{mathrsfs}
\usepackage[normalem]{ulem}
\swapnumbers
\usepackage[margin=3cm]{geometry}

\newtheorem{theorem}{Theorem}
\newtheorem{lemma}[theorem]{Lemma}

\newtheorem{proposition}[theorem]{Proposition}

\newtheorem{example}[theorem]{Example}

\theoremstyle{definition}
\newtheorem{remark}[theorem]{Remark}
\newtheorem{definition}[theorem]{Definition}

\numberwithin{figure}{section} 
\theoremstyle{plain}
\theoremstyle{plain}
\theoremstyle{remark}
\newtheorem*{acknowledgement*}{Acknowledgement}
\theoremstyle{example}

\newcommand{\cA}{{\mathcal A}}
\newcommand{\cB}{{\mathcal B}}

\newcommand{\cE}{{\mathcal E}}
\newcommand{\cF}{{\mathcal F}}
\newcommand{\cG}{{\mathcal G}}
\newcommand{\cH}{{\mathcal H}}
\newcommand{\cI}{{\mathcal I}}
\newcommand{\cJ}{{\mathcal J}}

\newcommand{\cL}{{\mathcal L}}

\newcommand{\cN}{{\mathcal N}}

\newcommand{\cP}{{\mathcal P}}

\newcommand{\cR}{{\mathcal R}}

\newcommand{\cX}{{\mathcal X}}

\newcommand{\Om}{{\Omega}}
\newcommand{\om}{{\omega}}
\newcommand{\ve}{{\varepsilon}}
\newcommand{\del}{{\delta}}

\newcommand{\sig}{{\sigma}}
\newcommand{\al}{{\alpha}}
\newcommand{\be}{{\beta}}
\newcommand{\ka}{{\kappa}}
\newcommand{\la}{{\lambda}}

\newcommand{\bbC}{{\mathbb C}}
\newcommand{\bbE}{{\mathbb E}}

\newcommand{\bbN}{{\mathbb N}}
\newcommand{\bbP}{{\mathbb P}}
\newcommand{\bbR}{{\mathbb R}}

\newcommand{\bbZ}{{\mathbb Z}}
\newcommand{\bbI}{{\mathbb I}}

\def\fg{\mathfrak{g}}

\DeclareMathOperator{\ess}{ess}
\DeclareMathOperator{\Var}{Var}

\newcommand{\DS}{\displaystyle}
\newcommand{\eps}{{\varepsilon}}

\begin{document}
\title[]{Generalized Edgeworth expansions for integer-valued additive functionals of uniformly elliptic Markov chains}
 \vskip 0.1cm
\author{Dmitry Dolgopyat and Yeor Hafouta}
\vskip 0.1cm
\address{
The University of Maryland and Ben-Gurion University of the Negev}

\dedicatory{  }
\maketitle

\begin{abstract}
We obtain asymptotic expansions for   probabilities $\bbP(S_N=k)$ 
of partial sums of uniformly bounded integer-valued functionals 
$\DS S_N=\sum_{n=1}^N f_n(X_n)$ of uniformly elliptic inhomogeneous  Markov chains.  The expansions involve products of polynomials and trigonometric
polynomials, and they hold  without additional assumptions.
As an application of the explicit formulas of the trigonometric polynomials, we 
 relate existence of the standard Edgeworth expansions of order $r$ to the rate of equidistributions
of $S_N$ modulo $m$ for small positive integers $m.$
\end{abstract}

\section{Introduction}
\subsection{Background}
Let $(X_n)$ be an inhomogeneous Markov chain.  In this paper we are interested in limit theorems for additive functionals of the form
$\DS S_N=\sum_{n=1}^N f_n(X_n)$, assuming that $\text{Var}(S_N)\to\infty$.
In 1956 Dobrushin \cite{Dob} proved  a
central limit theorem (CLT) for additive functionals of inhomogeneous 
contracting Markov chains. We refer to \cite{SV, Pel} for two modern approaches.

The  local central limit theorem (LCLT) has origins in the De Moivre-Laplace theorem, and it states that, uniformly in 
$k\in\bbZ$ we have
\[
\bbP(S_N=k)=\frac1{\sqrt{2\pi}\sig_N}e^{-\left(k-\bbE(S_N)\right)^2/2\sig^2_N}+o(\sig_N^{-1})
\]
where $\sig_N=\sqrt{\text{Var}(S_N)}$.   
The LCLT has applications to various areas of mathematics including 
 statistics \cite{Hall}, 
 mathematical physics
\cite{BLRB, DN16, DN-Mech, DT77, EM22, Kh49, LPRS, PS23}, number theory \cite{ADDS, Beck10, DS17}, 
geometry (\cite{LL15}), PDEs \cite{HKN}, and combinatorics (\cite{Can76, BR83, GR92, Hw98}).
Applications to dynamics include mixing properties of covers (\cite{AsymExp, DNP, OP19}),
suspensions flows (\cite{DN-llt}), skew products (\cite{Br05, DSL16, DDKN1, DDKN2, LB06}), and homogeneous dynamics (\cite{BeQu, BRLLT05, BRLLT23, Guiv2015, Hough2019}).

A related classical result for independent summands  $Y_n=f_n(X_n)$ is the {\em stable local central limit theorem} (SLCLT).  Recall that  $\DS S_N=\sum_{n=1}^N Y_n$ obeys the SLCLT if the LCLT holds true for any integer-valued square integrable independent sequence 
$Y_1',Y_2',\dots $ which differs from $Y_1, Y_2,\dots $ by a finite number of elements.  
 The idea of the SLCLT is to rule out "accidental" appearance of the LCLT (such as one given by \eqref{OddEven} below).
We recall a classical result due to Prokhorov\footnote{We refer the readers to \cite{Rozanov, VMT} for extensions 
Theorem \ref{ThProkhorov} to the case when $Y_n$'s are not necessarily bounded (for instance, the result holds true  when 
$\DS \sup_n\|Y_n\|_{L^3}<\infty$). 
Related results for local convergence to more general 
limit laws are discussed in \cite{D-MD, MS74}.
 See also \cite{Dolg16}, \cite[Theorem 6.5]{DS} for extensions
for real valued sums.
}.
\begin{theorem}
\cite{Prok} 
\label{ThProkhorov}
If $Y_n$ are independent  and uniformly bounded
then 
the SLCLT holds if and only if
for each  integer $h>1$,
\begin{equation}\label{Prokhorov}
 M_N(h):=\sum_{n=1}^N \bbP(Y_n\neq m_n \mod h)\to \infty \text{ as } N\to\infty,
\end{equation}
where $m_n=m_n(h)$ is the most likely residue of $Y_n$ modulo $h$.
\end{theorem}

The first  LCLT for homogeneous Markov chains was obtained in \cite{Nag}, and we refer to \cite{HH} for a general approach. Despite the fact that results in the  homogeneous case were already known in the 50's, only very recently (see \cite{DS}) the general case of  inhomogeneous (uniformly elliptic) Markov chains was solved (see also \cite{MPP,MPP1}). It turns out that the theory of the LCLT in the inhomogeneous case is richer than the  homogeneous case. In particular, one of the  main problems in the inhomogeneous
setting 
arises from  a possibility of non-linear growth of $\text{Var}(S_N)$ and the variances of the correspoding sums with $Y_n$ replaced by $Y_n$\text{ mod }h.  
We refer to \cite{DS} for a detailed discussion about the obstructions for the LCLT 
(in both lattice and non-lattice cases). Concerning the SLCLT, in this paper we adapt the methods of \cite{DS} and provide a version in the Markovian case (Theorem \ref{ThProkhorovMC}), see also next section.

The local limit theorem deals with approximation of $\bbP(S_N=k)$ up to an error term of order $o(\sig_N^{-1})$. In this paper, for uniformly bounded integer-valued functions $Y_n=f_n(X_n)$ of uniformly elliptic inhomogeneous Markov chains $\{X_n\}$, we  describe  a more refined type of approximations of the probabilities $\bbP(S_N=k)$. 
\begin{definition}
Given $r\geq1$,  we say that $S_N$ obeys the 
{\em Edgeworth expansion} (in the LCLT) of order $r$ 
if there are polynomials $P_{b, N}$, whose coefficients are uniformly bounded in $N$ and  whose degrees do not depend on $N$, so that 
 uniformly in $k\in\bbZ$,
\begin{equation}\label{EdgeDef}
\bbP(S_N=k)=\sum_{b=1}^r \frac{P_{b, N} (k_N)}{\sigma_N^b}\fg(k_N)+o(\sigma_N^{-r})
\end{equation}
where $k_N=\left(k-\bbE(S_N)\right)/\sig_N$ 
and  $\fg(u)=\frac{1}{\sqrt{2\pi}} e^{-u^2/2}. $    
\end{definition}

During the 20th century, the work of many authors led to the development of asymptotic expansions in both  the CLT and  the LCLT, 
 see \cite{Es, IL, Feller, BR, Pet, Hall2016} and references therein for more details.
In the past years there has been a growing interest in Edgeworth expansions  for classes of stationary weakly dependent random variables like Markov chains, chaotic dynamical systems and products of random matrices, see, for instance, \cite{FL, JWZ, FP}. In particular in \cite[Theorem 2.7]{FL} Edgeworth expansions in the LCLT for integer valued stationary sequences were obtained,  with notable applications for integer valued additive functionals of stationary elliptic Markov chains \cite[\S 6.3]{FL}. 
 A natural question is to determine
under which conditions Edgeworth expansions are true  for the corresponding non-stationary sequences. In \cite{DH BE} 
we obtained optimal sufficient conditions for Edgeworth expansions in the CLT of an arbitrary order, for several classes of additive functionals of uniformly elliptic inhomogeneous Markov chains\footnote{\cite{DH BE} deals with 
the non-lattice case. In the lattice case even the first order CLT expansion fails.
In the stationary case the non-lattice condition is also sufficient for the first order expansion,
see \cite{Es, HP},
while the non-stationary analogue is more complicated, see \cite{DH BE}.
}
  (some of the conditions were new even in the stationary case).
However, in the non-stationary setting
Edgeworth LCLT expansions are not always valid, see \cite[Section 10]{DH} for examples. 
 Moreover, for Markov chains, the validity of Edgeworth expansions may depend on the initial distribution. Indeed:
 \begin{example}[Validity of Edgeworth expansions]\label{Eg1}
 Suppose that $\{Y_n\}_{n\geq 2}$ are iid random variables taking values $\pm 1$ with probability $1/2$ while
$Y_1$ is independent of $\{Y_n\}_{n\ge 2}$ and takes values $0$ and $1.$ Take $r\geq 2$. Then the Edgeworth expansion of order $r$ is valid iff  
\begin{equation}
\label{OddEven}
\mathbb{P}(Y_1=0)=\mathbb{P}(Y_1=1)=\frac{1}{2}
\end{equation}
This fact follows easily by conditioning on $Y_1$ and 
applying the classical Edgeworth expansion for iid sum $\DS \sum_{n=2}^N Y_n$.
By contrast in the iid integer case Edgeworth expansions are always valid.
 \end{example}

\subsection{Our results}
To address a possible failure of Edgeworth expansions in the inhomogeneous case
we introduced
in \cite{DH} generalized expansions. For bounded
independent random variables $Y_n$, we obtained 
generalized trigonometric expansions for the probabilities $\bbP(S_N=k)$.
In this paper we will generalize the trigonometric expansions in \cite[Theorem 2.1]{DH} to the inhomogeneous Markov case (see Theorem \ref{MainThm}).
 The key difficulty of handling the dependent case is that Edgeworth expansions 
rely on asymptotic analysis of characteristic functions. For dependent random variables such 
expansions use the spectral method of Nagaev \cite{Nag}, see \cite{HH}. This approach was extended to non stationary
setting in \cite{HK}, however, it only works in the perturbative regime where the argument of the characteristic
function is close to 0. While it is possible to show that the failure of Edgeworth expansions implies that 
there are large blocks where the twisted transfer operator could be conjugated into the perturbative region, in 
the non stationary setting there could be a small number of exceptional terms where the above reduction
does not work. To address this difficulty we introduce in this paper conditional Edgeworth expansions, see Definition \ref{CondExpDef}.

Next,  in Theorem \ref{ThProkhorovMC} we will generalize Theorem \ref{ThProkhorov} to the Markovian setting.
Similarly to \cite{DH}, as a part of the analysis in Theorem \ref{MainThm} (ii) we provide a quantitative version of Theorem \ref{ThProkhorov} which provides natural sufficient conditions for the regular expansions of a given order $r$ to hold  in terms of the growth of $M_N(h)$ for small integer~$h$.

Another  application is a  generalization of \cite[Theorem 1.8]{DH} to  integer-valued additive functionals of uniformly
elliptic Markov chains. That is, in Theorem \ref{Thm Stable Cond} we give a complete characterization for  conditional expansions of an arbitrary order $r$. 
 Our conditions (see part $(b)$ in Theorem \ref{Thm Stable Cond}) are formulated in terms of rates of decay of the (conditioned) characteristic functions of $S_N$  at nonzero ``resonant points" of the form  $t=\frac{2\pi l}{m}$, with $0<m\leq 2K$
and $0\leq l<m$, where $\DS K:=\sup_{n}\|Y_n\|_{L^\infty}<\infty$.  Equivalently  
 (see Theorem \ref{Thm Stable Cond} (c)), they are formulated by means of the decay rates of the distance of the (conditional) distribution of 
$\DS S_N-\sum_{j=1}^\ell Y_{k_{j,N}}$ modulo $m$  from the uniform distribution, for all $m$, and an arbitrary choice of indexes $k_{1,N},...,k_{\ell,N}$. 
Part (b) of Theorem \ref{Thm Stable Cond} provides a simpler analytic condition, while part (c) of Theorem \ref{Thm Stable Cond} is more algebraic in spirit.
Note that  both conditions in part (b) and part (c) discussed above are simpler than the Edgeworth expansion themselves. 
 In particular they require checking finitely many conditions as opposed to \eqref{EdgeDef} which
needs to be verified for each integer $k.$ In fact,  the distribution of $S_N$ modulo $m$ reduces to studying
finite state Markov chains which could be analyzed in many applications. In particular,  
\cite[Section 10]{DH} contains various examples of applications of Theorem \ref{Thm Stable Cond} to sums of independent
random variables which are close to being concentrated on cosets of $m\mathbb{Z}.$ The present work
allows us to analyze similar examples in the Markov case.


 Finally, let us point out that our interest in non stationary
systems is motivated among other things by applications to renormalization
(cf.  \cite{ARHW22, ARHW23, ADDS,  BU, AsymExp}) and to random walks in random environment
(\cite{BCR16, CD24, DG13, DG21, DK20}). Additionally, 
 our results are partially motivated by applications  to ergodic theory.
In a follow up paper we plan to apply our expansions to obtain asymptotic expansions of ergodic
integrals for translation surfaces, extending the results of  \cite{AsymExp} beyond the pseudo Anosov case.

\section{Main results}
\label{ScMain}
\subsection{Ellipticity and mixing conditions}
 Let $\{X_n: n\in\bbN\}$ be a Markov chain  such that each $X_n$ takes values on some countably generated measurable space. Let us denote by $\mu_n$ the law of $X_n$. We  assume that there is a constant $C>1$ such that for  all $n$, for $\mu_n$-a.e. $x$ and for all 
 measurable subsets $A$ of the state space of $X_{n+1}$ we have
\begin{equation}\label{Doeblin}
C^{-1}\bbP(X_{n+1}\in A)\leq P(X_{n+1}\in A|X_{n}=x)\leq C\bbP(X_{n+1}\in A).
\end{equation}
This condition is equivalent to the following representation of the transition probabilities:
\begin{equation}\label{Densities}
\bbP(X_{n+1}\in A|X_{n}=x)=\int_{A}p_n(x,y)d\mu_{n+1}(y)
\end{equation}
where 
the transition densities $p_n(x,y)$ are measurable functions taking values in the interval $[C^{-1},C]$.
We note that \eqref{Densities} implies that  $\{X_n: n\in\bbN\}$ is  exponentially fast $\psi$-mixing,
see e.g. \cite[Proposition 1.22]{DS}.
Namely, there are  constants $C_1>0$ and $\del\in(0,1)$ so that if $\bar X$ is a measurable  function of $X_1,...,X_m$ and $\bar Y$ is a measurable function of $X_{m+n},X_{m+n+1},...$ for some $m$ and $n$, then for every relevant measurable sets $A,B$, 
\begin{equation}\label{psi}
\left|\bbP(\bar X\in A , \bar Y\in B)-P(\bar X\in A)P(\bar Y\in B)\right|\leq C_1 P(\bar X\in A)P(\bar Y\in B)\del^n.
\end{equation} 

\subsection{Generalized expansions}
For each $n$, let $f_n$ be  a measurable integer valued function on the state space of $X_n$, and set $Y_n=f_n(X_n)$. We assume  that $\DS K:=\sup_n\|Y_n\|_{L^\infty}<\infty$.
Let 
\begin{equation}\label{M deff}
M_N:=\min_{2\leq h\leq 2K}\sum_{n=1}^N\bbP(Y_n\neq m_n(h) \text{ mod } h)
\end{equation}
where $m_n(h)$ is the most likely residue of $Y_n$ modulo $h$ (like in Theorem \ref{ThProkhorov}). Our first result is the following trigonometric expansions.

\begin{theorem}\label{MainThm}
Let $S_N=Y_1+Y_2+...+Y_N$ and $\sig_N=\sqrt{\text{Var}(S_N)}$.

(i) There is a positive integer  $J=J(K)$, which depends only on $K$,  and polynomials $P_{a, b, N}$ with degrees depending only on $a$ and $b$, whose coefficients are uniformly bounded in $N$, such that, for every $r\geq1$, uniformly in $k\in\bbZ$, we have
$$\bbP(S_N=k)-\sum_{a=0}^{J-1} \sum_{b=1}^r \frac{P_{a, b, N} ((k-a_N)/\sigma_N)}{\sigma_N^b}
\fg((k-a_N)/\sigma_N) e^{2\pi i a k/J} =o(\sigma_N^{-r})
$$
where $a_N=\bbE(S_N)$ and $\fg(u)=\frac{1}{\sqrt{2\pi}} e^{-u^2/2}. $

(ii) 
There exists a constant $R=R(K,r)$, depending only on $K$ and $r$, such that if 
$M_N\geq R \ln \sigma_N$ then we can choose $P_{a, b, N}=0$ for $a\neq 0$ (that is the regular expansions of order $r$ hold). 

In particular, $S_N$ obeys the Edgeworth expansion of all orders if 
$$
\lim_{N\to\infty}\frac{M_N}{\ln\sig_N}=\infty.
$$
\end{theorem}

As noted in the previous section, in the inhomogeneous case the regular Edgeworth expansion might fail (see Example \ref{Eg1}). The first part of the theorem provides the right correction terms in the local limit theorem. The double sum  comes from expanding the characteristic function around resonant points (see Definition \ref{DefResP}), and the sum corresponding to $a=0$ is the contribution of the resonant point $0$.
The second part  is a quantitative version of Prokhorov's theorem for Markov chains, and its proof appears in \S \ref{Sec4}.
In \cite{DH BE} we showed, in particular, that $S_N$ obeys the Edgeworth in the CLT of  order $r<\frac{1+\al}{1-\al}$ if the functions $f_n$ are uniformly H\"older continuous with respect to an exponent $\al\in(0,1]$.  Compared with \cite{DH BE}, Theorem~\ref{MainThm}(ii) provides sufficient conditions for high order expansions in the LCLT for integer-valued functionals $f_n$ (which, in general, are not even continuous). 
We observe that logarithmic in $\text{Var}(S_N)$ growth of various non-periodicity characteristics
of individual summands are often used in the theory of local limit theorems
(see e.g. \cite{Mal78, MS70, MPP}).  
However, to  justify the optimality we need to understand the conditions necessary 
for the validity of the Edgeworth expansions. 



\begin{remark}
 Note that if 
$$M_N(h)=\sum_{n=1}^N\bbP(Y_n\neq m_n(h))\leq R(r,K,\ve_0)\ln\sig_N$$
then for most $n,$  the distributions of $Y_n$ are sufficiently close to being concentrated on a single point modulo $h$.
Our arguments also show that if 
this closeness holds for {\em all} $n$,
  then the Edgeworth expansions of order $r$ is valid 
iff  $\bbE[e^{itS_N}]=o(\sig_N^{-(r-1)}),$ for every nonzero $K$--resonant point $t$ (see Definition \ref{DefResP}). In the general case we have Theorem~\ref{Edge} below which shows that  $\bbE[e^{itS_N}]=o(\sig_N^{-(r-1)})$ is always a necessary condition.

 We also note that if the distribution of $Y_n$ mod $m$ is not approximately
concentrated  on a single point, the condition $\bbE[e^{it S_N}]=o(\sig_N^{-(r-1)})$ does 
not imply that the Edgeworth expansion holds even in the independent case, see
\cite[Example 10.2]{DH}. In fact, in the independent case, if one of the $Y_n$'s is uniformly distributed modulo $m$, then $\bbE[e^{it S_N}]=0$
for all $N$ large enough, and every nonzero $K$--resonant point of the form $t=\frac{2\pi l}{m}$. However, this does not imply that the derivatives of the characteristic function 
vanish at $t$, and hence by \cite[Theorem 1.5]{DH} expansions of an arbitrary order $r$ might not hold.
\end{remark}



\subsection{Conditional expansions}
In this section we will generalize some of the results in \cite{DH} from independent summands to inhomogeneous  Markov chains.
\begin{definition}\label{CondExpDef}
 We say that $\DS S_N=\sum_{n=1}^Nf_n(X_n)$ obeys the \textit{conditional Edgeworth expansion} of order $r$ 
  if  \eqref{EdgeDef} holds for $\bbP(S_N=k|X_{j_1},...,X_{j_\ell})$ instead of $\bbP(S_N=k)$ for any choice of indexes $j_1<j_2<...<j_\ell$ with error terms $o(\sig_N^{-r})$ not depending on $X_{j_1},...,X_{j_\ell}$ and with polynomials that depend on $X_{j_1},...,X_{j_\ell}$, but that have uniformly bounded coefficients and whose degrees do not depend on $N$ and $X_{j_1},...,X_{j_\ell}$.    
\end{definition}

\begin{definition}
\label{DefResP}
We say that a number $t$ is  \textit{$K$--resonant} if $t=\frac{2\pi l}{m}$ for some $0<m\leq 2K$
and $0\leq l<m.$
\end{definition}

\begin{theorem}\label{Thm Stable Cond}

Let $f_n$ be integer-valued functionals such that 
$\DS K\!\!:=
\!\sup_n\|f_n(X_n)\|_{L^\infty}\!\!<\!\!\infty.$
For an arbitrary integer $r\geq 1$ the following conditions are equivalent:

(a) $\DS S_N=\sum_{n=1}^{N}f_n(X_n)$ obeys the conditional Edgeworth expansions of order $r$.

(b) For each $\ell$, and all nonzero $K$--resonant points $t$,
$$
\sup_{1\leq j_1,...,j_\ell\leq N}\left\|\bbE[e^{itS_N}|X_{j_1},...,X_{j_\ell}]\right\|_{L^1}=o(\sig_N^{-(r-1)}).
$$

(c) For every choice of a positive integer $\ell$, indexes $1\leq j_1,...,j_\ell\leq N$ and an integer   $2\leq h \leq 2K$, 
the conditional distribution of $S_N$ given $X_{j_1},...,X_{j_\ell}$ mod $h$ is
$o(\sigma_N^{1-r})$ close to uniform, uniformly in the choice of $j_1,...,j_\ell$.
\end{theorem}
 This result extends \cite[Theorem 1.8]{DH} (which dealt with independent summands).
The proof of Theorem \ref{Thm Stable Cond} appears in Section \ref{Sec6}. We refer to the penultimate paragraph of the previous section for a discussion about conditions $(b)$ and $(c)$.

In the course of the proof of Theorem \ref{Thm Stable Cond} we obtain the following result.

\begin{theorem}\label{Edge}
(i) 
If $S_N$ obeys the classical Edgeworth expansions of order $r$ (i.e. \eqref{EdgeDef}), then $\bbE[e^{it S_N}]=o(\sig_N^{-(r-1)})$ for all $K$--resonant points $t$.
\vskip0.2cm
(ii) There is a natural number $\ell_r$ which depends only on $r$ so that the condition 
$$
\max_{\ell\leq \ell_r}\sup_{ 1\leq j_1,...,j_\ell\leq N}\left\|\bbE[e^{itS_N}|X_{j_1},..,X_{j_\ell}]\right\|_{L^1}=o(\sig_N^{-(r-1)})
$$
for all $K$--resonant points $t$ is sufficient for the classical Edgeworth expansions of order $r$. 
\end{theorem}

The number $\ell_r$ in part (ii) can be recovered from the proof of the theorem
(for instance, we have $\ell_2=6$).

The equivalence between $(a)$ and $(b)$ in Theorem \ref{Thm Stable Cond} follows from  Theorem \ref{Edge}
since we can apply it to the conditional law of $S_N$ given a finite number of $X_j$'s, and in part (i) of Theorem \ref{Edge} the term $o(\sig_N^{-(r-1)})$ will depend only on the error term of the Edgeworth expansions.

\begin{remark}\label{R5}
In principle, first order expansions and the local central limit theorem might differ since the polynomial $P_{1,N}$ does not have to equal $1$ identically. However, 
it will follow from the proof of Theorem \ref{Thm Stable Cond} in the case $r=1$ that the conditional local central  limit
theorem (namely the local limit theorem after conditioning on a finite number of elements) is equivalent to the conditional  Edgeworth expansion of order $1$. Thus,   the conditional  local limit theorem holds true iff  for every $\ell$,
$$
\sup_{1\leq j_1,...,j_\ell\leq N}\left\|\bbE[e^{itS_N}|X_{j_1},..,X_{j_\ell}]\right\|_{L^1}=o(1).
$$
That is, iff (c) holds with $r=1$, 
namely, for all integers $h$ and $L$ we have the following: for every  sequence 
$(j_{1,N}, \dots ,j_{\ell_N, N})$ of tuples with $\ell_N\leq L$  the conditional distributions
of $(S_N|X_{j_1,N}, \dots, X_{j_{\ell_N},N})$ mod $h$ converge to uniform distribution, uniformly in $L$,   as $N\to\infty$. Equivalently
for each $m\in \bbZ$,
$$ \lim_{N\to\infty} \bbE[e^{imS_N/h}|X_{j_1,N},..,X_{j_\ell,N}]=0 $$
and the  convergence is uniform in the choice of indexes.
\end{remark}

We also have note a simpler version of the conditional local central limit theorem  in the theorem  below.

\begin{theorem}
\label{ThProkhorovMC}
The conditional local central limit theorem holds iff $\DS \lim_{N\to\infty} M_N=\infty.$
\end{theorem} 
Theorem \ref{ThProkhorovMC}  is obtained by  a small modification of the proof of \cite[Theorem 6.5]{DS}
as we explain in \S \ref{6.1}.

\subsection{A discussion about the conditions and open  problems}
 The results in this  paper are obtained under 
the uniform ellipticity condition \eqref{Doeblin}, which  ensures that the Markov chain $\{X_n: n\in\bbN\}$ is exponentially fast $\psi$-mixing (i.e. \eqref{psi} holds). Note that by Lemma \ref{psi dens lem},  the $\psi$-mixing rates determine the speed that the transition densities of $X_{n+k}$ given $X_n$ converge to $1$ as $k\to\infty$.

A key ingredient in our proofs are certain conditioning arguments (see \S\S \ref{Sec 3.5} and \ref{Sec 5.1}). A key step in our proof is to condition on $X_n, n\in\cE$ for  sets $\cE$ of ``bad" indexes such that  $X_n\text{ mod }m, n\in \cE$ does not have a sufficiently small  variance. In this step \eqref{Doeblin}  is needed in the proof of Lemma \ref{FiberQuantLemma} which ensures that quantities like $M_N$ (defined in \eqref{M deff}) are of the same order after such a  conditioning. This is what allows us to formulate the results by means of $M_N$ and not some random conditioned versions of them. Condition \eqref{Doeblin} is also used in the proof of Lemma \ref{CondChain} which tells us that after such a conditioning each independent block also satisfies \eqref{Doeblin}, possibly with different constants. This lemma is then used in the proof of Lemma \ref{L1} which shows that both conditional expectation and variances of the underlying partial sum do not change too much after the conditioning.

Another key ingredient in our analysis is  the 
sequential complex Perron-Frobenius theorem (see Theorem \ref{RPF}), and 
\eqref{Doeblin} is also used to prove this result (also after the above conditioning). We believe that Theorem \ref{RPF} is still true under  the weaker exponential $\phi$-mixing, but it is unclear if we can get a similar results under weaker mixing conditions (say, $\al$-mixing). 

It is  an interesting open problem to extend the main results in the paper to more general mixing  systems, where already the case of $\phi$-mixing Markov chains seems to require new ideas.

\section{Auxiliary results}
\subsection{A generalized sequential Perron-Frobenius theorem}\label{RPF sec}
Let $B_j$ denote the space of bounded 
measurable functions of $X_j$, and let $\|\cdot\|_{\infty}$ be the supremum norm. Let $B_j^*$ denote the dual space of $B_j$, that is the space of all continuous linear functionals on $B_j$, equipped with the operator norm 
$\DS \|\mu\|=\sup_{g\in B_j, \|g\|_\infty\leq 1}|\mu(g)|$. Let us also denote by $\textbf{1}$ the function that takes the constant value $1$, 
regardless of its domain.

Consider measurable real-valued functions $U_n=u_n(X_n)$ 
such that $\DS \sup_n\|U_n\|_{L^2}<\infty$ and for every complex number $z$ consider the operator $R_{z}^{(j)}: B_{j+1}\mapsto B_{j}$ 
defined by
$$
R_{z}^{(j)}g(x)=\bbE[e^{iU_{j+1}+zY_{j+1}}g(X_{j+1})|X_j=x].
$$
For each $j$ and $n$ in $\bbN$ let 
$$
R_z^{j,n}=R_{z}^{(j)}\cdots R_z^{(j+n-1)}.
$$

 The next result proven in Appendix \ref{App B}
is  one of our key technical tools.
\begin{theorem}\label{RPF}
There exists a number $\del_0>0$ which depends only on the
uniform bound $K$ of  $Y_n$ and on the ellipticity constant $C$ of $X_n$ so that
 the following holds.
If $\DS \sup_n\|U_n\|_{L^2}\!+\!|z|\!<\!\del_0$
 then 
for every $j\in\bbZ$ there exists a  triplet
$\la_j(z)$, $h_j^{(z)}$ and $\nu_j^{(z)}$ consisting of a nonzero complex number
$\la_j(z)$, a  complex function $h_j^{(z)}\in B_{j}$ and a 
continuous linear functional $\nu_j^{(z)}\in B_j^*$ satisfying  $\nu_j^{(z)}(\textbf{1})=1$, $\nu_j^{(z)}(h_j^{(z)})=1$,
\[
R_z^{(j)}h_{j+1}^{(z)}=\la_j(z)h_j^{(z)}
,\,\,\text{ and }\,\,(R_z^{(j)})^*\nu_{j}^{(z)}=\la_j(z)\nu_{j+1}^{(z)}
\]
where $(R_z^{(j)})^*:B_{j}^*\to B_{j+1}^*$ is the dual operator of $R_j^{(z)}$.
 When $z=0$ and $U_j=0$ we have $\la_j(0)=1$ and $h_j^{(0)}=\textbf{1}$.

Moreover, with $u=\sup_n\|U_n\|_{L^2}$,
$\DS
\sup_{j}\sup_{|z|<\del_0-u}\max\left(|\la_j(z)|, \|h_j^{(z)}\|_\infty, \|\nu_j^{(z)}\|_{\infty}\right)<\infty
$
and the maps $z\to \la_j(z)$, $z\to h_j^{(z)}$ and $z\to \nu_j^{(z)}$ are 
 holomorphic.
In addition, $\la_j(z),h_j(z)$ and $\nu_j(z)$ depend continuously on $U_j$ in the sense that they converge uniformly to the triplets  corresponding to the choice $U_j=0$ as 
$\DS \sup_n\|U_n\|_{L^2}\to 0$. 

Furthermore, there exist  constants $C>0$  and $\del\in(0,1)$ such that for all 
$n\geq1$, $j\in\bbZ$ and $q\in B_{j+n}$,
\begin{equation}\label{Exp Conv final.0.1.1}
\bigg\|\frac{R_z^{j,n}q}{\la_{j,n}(z)}
-\big(\nu_{j+n}^{(z)}(q)\big)h_j^{(z)}\bigg\|_\infty\leq\\
C\|q\|_\infty\cdot \del^n
\end{equation}
where $\DS \la_{j,n}(z)=\prod_{k=0}^{n-1}\la_{j+k}(z)$.
\end{theorem}

\subsection{Behavior of the characteristic functions near $0$}
Let
$$
\Lambda_{N}(h)=\ln\bbE[e^{ih(S_N-\bbE[S_N])/\sig_N}]+h^2/2.
$$
 By \cite[Proposition 30]{DH BE} for every $m$ there exist constants $\del_m,C_m>0$ so that for all $j\geq 3$, 
\begin{equation}\label{Growth0}
    \sup_{h\in[-\del_m\sig_N,\del_m\sig_N]}|\Lambda_{N}^{(j)}(h)|\leq C_m\sig_N^{-(j-2)}.
\end{equation}
Set
$$
Q_{r,N}(t)=\sum_{\bar k}\frac1{k_1!\cdots k_{r}!}\left(\frac{\Lambda_{N}^{(3)}(0)}{3!}\right)^{k_1}\cdots \left(\frac{\Lambda^{(r+2)}_{N}(0)}{(r+2)!}\right)^{k_{r}}(it)^{3k_1+...+(r+2)k_{r}}
$$
where the summation ranges over the collection of $r$ tuples of nonnegative integers $(k_1,...,k_{r})$ that are not all $0$ so that $\DS \sum_{j} jk_j\leq r$.
Then 
\begin{equation}\label{Q}
Q_{r,N}(t)=\sum_{j=1}^{r}\sig_N^{-j}P_{j,N}(t)
\end{equation}
with 
\begin{equation}\label{P}
P_{j,N}(x)=\sum_{\bar k\in A_{j}}C_{\bar k}\left(\prod_{j=1}^{s}\left(\sig_N^{-2}\Lambda_N^{(j+2)}(0)\right)^{k_j}\right)(ix)^{3k_1+...+(s+2)k_{s}}, 
\end{equation}
where $A_{j}$ is the set of all  tuples of nonnegative integers $\bar k=(k_1,...,k_{s}), k_s\not=0$ for some $s=s(\bar k)\geq 1$ so that $\DS \sum_{s}sk_s=j$ (note that when $j\leq r$ then $s\leq r$ since $k_s\geq1$). Moreover
 $\DS
 C_{\bar k}=\prod_{j=1}^{s}\frac{1}{k_j!(j+2)^{k_j}}.
 $

\begin{lemma}(\cite[Section 4.3]{DH BE}).\label{L}
Let $W_N=(S_N-\bbE[S_N])/\sig_N$.
For every $r\geq1$ there are constants $\del_r,C_r>0$ so that for every $t\in[-\del_r\sig_N,\del_r\sig_N] $ we have
$$
\left|\bbE[e^{itW_N}]-e^{-t^2/2}(1+Q_{r,N}(t))\right|\leq Ce^{-ct^2}\sig_N^{-(r+1)}
\max\left(|t|, |t|^{(r+3)(r+2)}\right)
$$
where $c>0$ is a constant  independent of $r$ (and, by decreasing $\del_r$, 
 it can be  made arbitrarily close to $1/2$).
\end{lemma}

We note that the estimate in Lemma \ref{L} appeared in 
\cite[\S 4.3]{DH BE} at the end of the proof of Proposition 29, but it was not explicitly formulated as an auxiliary result in \cite[\S 4.3]{DH BE}.

\subsection{Behavior of the characteristic functions near resonant points: a key tool}
 Lemma \ref{L} will be crucial to determine the contribution of the resonant point $0$. To determine the contribution of other non-resonant points we will also need the following more general result, whose proof proceeds exactly as the proof of \cite[Proposition 23]{DH BE} 
 (but, for the sake of completeness, the full details are given in Appendix \ref{Appendix A}).

\begin{proposition}\label{LamProp}
Fix some integer $r\geq 1$.
Let $\cL_N:\bbR\to\bbC$ be an $r+2$ times differentiable function  so that 
$$
\cL_N(0)=\cL_N'(0)=\cL_N''(0)=0
$$
and that for each $3\leq j\leq r+2$ for every $t\in[-\del_r,\del_r]$
we have
\begin{equation}\label{Growth}
  \left|\cL_N^{(j)}(t)\right|\leq A_r\sig_N^{2}  
\end{equation}
where $\del_r$ and $A_r$ are constants which do not depend on $t$ and $N$. Set 
$
\DS \bar\cL_N(t)=\cL_N(t/\sig_N).
$
Then there are constants $B_r,\ve_r>0$ depending only on $\del_r$ and $A_r$ so that for every $t\in[-\ve_r\sig_N,\ve_r\sig_N]$ we have
$$
\left|e^{\bar \cL_N(t)}-\cH_{N,r}(t)\right|\leq
e^{\frac14 t^2} B_r(\sig_N)^{-(r+1)}|\max(|t|, |t|^{(r+2)(r+3)})
$$
where 
\begin{equation}\label{H def0}
\cH_{N,r}(t)=
1+\sum_{\bar k}\frac1{k_1!\cdots k_{r}!}\left(\frac{\bar\cL_{N}^{(3)}(0)}{3!}\right)^{k_1}\cdots \left(\frac{\bar\cL^{(r+2)}_{N}(0)}{(r+2)!}\right)^{k_{r}}t^{3k_1+...+(r+2)k_{r}}
\end{equation}
and the summation runs over the collection of $r$ tuples of nonnegative integers 
$(k_1,...,k_{r})$ that are not all $0$ so that $\DS \sum_{j} jk_j\leq r$. 
\end{proposition}
\begin{remark}
In view of \eqref{Growth0}, 
 Proposition \ref{LamProp} actually implies Lemma \ref{L}, but to increase readability we decided to  formulate Lemma \ref{L} as a separate result.
 \end{remark}

\subsection{Mixing properties and moment estimates}
\begin{lemma}\label{MomEst}
For each $j$, let $G_j=g_j(X_j)$ be a real valued measurable function of $X_j$ so that 
$\DS \|G\|_\infty:=\sup_j\|G_j\|_{L^\infty}<\infty$. For every $k,n\in\bbN$ such that $k\leq n$ let 
$\DS G_{k,n}\!\!=\!\!\sum_{j=k}^{n}G_j$. Then
  \vskip0.1cm

\noindent
(i) There are constants $C_1,C_2>0$ depending only on the ellipticity constant $C$ in \eqref{Doeblin} such that
$$
C_1\sum_{j=k}^{n}\text{Var}(G_j)\leq\text{Var}(G_{k,n})\leq C_2\sum_{j=k}^{n}\mathrm{Var}(G_j).
$$
  \vskip0.1cm
 
 \noindent   
(ii) For every $p>2$ there is a constant $R_p$ depending only on $p,C$ and $\|G\|_\infty$ such that
$$
\left\|G_{k,n}-\bbE[G_{k,n}]\right\|_{L^p}\leq R_p\left(1+\sqrt{\mathrm{Var}(G_{k,n})}\right).
$$
\end{lemma}
The first part follows\footnote{Note that the results in \cite{Pel} are formulated for Markov chains with correlation coefficients uniformly smaller than $1$. However, this result can be applied in the uniformly elliptic setup of this paper since, as was shown in \cite{Pel}, such chains are uniformly contracting in the sense of Dobrushin \cite{Dob}, and the uniform one step ellipticity implies uniform contraction (see \cite[Lemma 2.12]{DS}).} 
from \cite[Proposition 13]{Pel}, while the  harder lower bound in the first estimate was obtained in \cite[Proposition 3.2 ]{SV}.  We also refer to  \cite[Theorem~2.1]{DS}, which  in the case of
one step elliptic Markov chains and functionals of the form $G_j$
reduces to these variance estimates of part (i).

The second result was essentially obtained in \cite[Lemma 2.16]{DS} and it also follows from \cite[Theorem 6.17]{MPU19}.

\begin{lemma}[Proposition  1.11 (2), \cite{DS}]\label{DEC}
Let $G_j=g_j(X_j)$ be a real valued function of $X_j$ so that 
$\DS \sup_j\|G_j\|_{L^\infty}<\infty$. Then there exist $\del\in(0,1)$ and $A>0$ which depend only on the ellipticity constant $C$ and on $\|G\|_\infty$ so that for all $n,k\in\bbN$ we have 
$$
\left|\text{\rm Cov}\big(g_n(X_n), g_{n+k}(X_{n+k})\big)\right|\leq A\del^k.
$$
\end{lemma} 
\begin{remark}
We note that by \eqref{psi} and  \eqref{PsiRep} we can get that 
$$
\left|\text{\rm Cov}\big(g_n(X_n), g_{n+k}(X_{n+k})\big)\right|\leq C_1\|g_n(X_n)\|_{L^1}\|g_{n+k}(X_{n+k})\|_{L^1}\del^k.
$$
However, this stronger estimate will not be used in this paper.
\end{remark}

\subsection{Mixing and transition densities}
Let $\mu_n$ denote the law of $X_n$.

\begin{lemma}\label{dens lemma}
We have 
 $\DS
\bbP(X_{n+k}\in A|X_n=x)=\int p_{n}^{(k)}(x,y)d\mu_{n+k}(y)
$
and the transition densities $p_{n}^{(k)}(x,y)$ take values in $[C^{-1},C]$, where $C$ comes from \eqref{Doeblin}.
\end{lemma}
\begin{proof}
We have 
$$
\bbP(X_{n+k}\in A|X_n)=\bbE[\bbP(X_{n+k}\in A|X_{n+k-1},X_n)|X_n]=
\bbE[\bbP(X_{n+k}\in A|X_{n+k-1})|X_n].
$$
To complete the proof, note that 
$\DS
\bbP(X_{n+k}\in A|X_{n+k-1})\in[C^{-1},C],\,\,\text{a.s.}
$
\end{proof}

In the course of the proofs we will also need the following result.
\begin{lemma}\label{psi dens lem} 
$\DS
\ess\sup_{x,y}\left|p_{n}^{(k)}(x,y)-1\right|\leq C_1\del^k
$
where the essential supremum is with respect to the joint law of $X_n$ and $X_{n+k}$ (and $C_1$ and $\delta$ come from \eqref{psi}).
\end{lemma}
\begin{proof}
Let $\Gamma(A)$ be a random signed measure given by 
$$
\Gamma(A)=\bbP(X_{n+k}\in A|X_n)-\bbP(X_{n+k}\in A).
$$
Then, $y\to p_{n}^{(k)}(X_n,y)-1$ is the Radon-Nikodym derivative $d\Gamma/d\mu_{n+k}$. Let $(\Om,\cF,\bbP)$ be a probability space  
on which the Markov chain $\{X_k: k\in\bbN\}$ is defined.
Recall that by \cite[Ch.4]{Brad}, for every two sub-$\sigma$-algebras $\cG,\cH$ of  $\cF$,
\begin{equation}\label{PsiRep}
\psi(\cG,\cH):=\sup\left\{\left|\frac{\bbP(A\cap B)}{\bbP(A)\bbP(B)}-1\right|: A\in\cG, B\in\cH, \bbP(A)\bbP(B)>0\right\}
\end{equation}
$$
=\sup\left\{\|\bbE[h|\cG]-\bbE[h]\|_{L^\infty}: h\in L^1(\Om,\cH,\bbP), \|h\|_{L^1}\leq 1\right\}.
$$
Let $\cG=\sig\{X_n\}$ and $\cH=\sig\{X_{n+k}\}$. Then, by condition \eqref{psi}, we have 
$\psi(\cG,\cH)\leq C_1\del^k$. Hence,
applying \eqref{PsiRep} with the function
 $h=\bbI(X_{n+k}\in A)/\bbP(X_{n+k}\in A)$, we see that
\begin{equation}\label{Estt}
    \left\|\bbP(X_{n+k}\in A|X_n)-\bbP(X_{n+k}\in A)\right\|_{L^\infty}\leq C_1\bbP(X_{n+k}\in A)\del^k.
\end{equation}
Recall our assumption that  the $\sigma$-algebra $\cF_{n+k}$ on the state space 
$\cX_{n+k}$ of $X_{n+k}$ is countably generated. Let $\{A_i: i\in\bbN\}$ be a countable generating set of measurable subsets of $\cX_{n+k}$. 
Then, by \eqref{Estt},\, $\bbP$-a.s. for every $i$ we have
$$
\left|\bbP(X_{n+k}\in A_i|X_n)-\bbP(X_{n+k}\in A_i)\right|\leq C_1\bbP(X_{n+k}\in A_i)\del^k.
$$
Denote by $\cA_n$ the finite  $\sigma$-algebra generated by $A_1,...,A_n$.  
Then given a probability measure $\ka$ on $\cX_{n+k}$, it is easy to show that the class of all measurable sets $A$ such that $\DS \lim_{m\to\infty}\ka(A\Delta B_m)=0$, for some sequence  $B_m\in \cA_m$, is a $\sig$-algebra. Since this class contains a generating set we conclude that every measurable subset of $\cX_{n+k}$ is a limit of such a sequence. Now, by taking $\ka$ be to either $A\mapsto\bbP(X_{n+k}\in A|X_n)$ or $A\mapsto\bbP(X_{n+k}\in A)$ we see that 
a given measurable set is a limit in probability of sets in $\cA_m$
(with respect to both $\mu_{n+k}$ and the conditional law of $X_{n+k}$, which is equivalent to $\mu_{n+k}$).
Using that $\cA_m$ are finite $\sig$-algebras we conclude that $\bbP$-a.s. for every measurable set $A$ in $\cX_{j+n}$ we have
$$
\left|\bbP(X_{n+k}\in A|X_n)-\bbP(X_{n+k}\in A)\right|\leq C_1\bbP(X_{n+k}\in A)\del^k.
$$
In other words, if $\mu_n$ is the law of $X_n$ then
$$
\sup_{A\in\cF_{n+k}}\left(\bbP(X_{n+k}\in A)\right)^{-1}\left|\bbP(X_{n+k}\in A|X_{n}=x)-\bbP(X_{n+k}\in A)\right|\leq C_1\del^k,\,\,\,\, \mu_n-\text{a.s.}
$$
Hence $d\Gamma/d\mu_{n+k}$ is bounded by $C_1\del^k$.
\end{proof}

\subsection{Conditioning}\label{Sec 3.5}
 Let $\cE$ be the $\sig$-algebra generated by $\{X_n: n\in\cB\}$ where $\cB$ 
is a subset of $\{1,\dots, N\}.$
\begin{lemma}\label{FiberQuantLemma}
(i) There is a constant $C_2\geq1 $ such that for every $n\not\in \cB$ 
we have
\begin{equation}\label{VarRel}
C_2^{-1}\text{Var}(Y_n)\leq \text{Var}(Y_n|\cE)\leq C_2\text{Var}(Y_n),\,\,\,
\text{a.s.}
\end{equation}

(ii) For every $n$, let $q_n(m|\cE)$ be the second largest value among $P(Y_n\equiv j\mod m|\cE)$, \\$j=0,1,\dots,  m-1$. 
Then there exists a constant $A\geq1 $ such that for every $n\not\in \cB$ 
we have
\begin{equation}\label{q-Rel}
A^{-1}q_n(m)\leq q_n(m|\cE)\leq Aq_n(m),\,\,\,\text{a.s.}
\end{equation}
where $q_n(m)$ is the second largest among $P(Y_n\equiv j\mod m)$.
\end{lemma}
\begin{proof}
In the course of the proof it will be convenient to assume that the given Markov chain $(X_n)_{j\geq 1}$ is actually a two sided uniformly elliptic Markov chain. Note that, once we are able to extend the chain to a two sided one, the estimates in Lemma \ref{FiberQuantLemma} for the one sided chain will follow from the estimates obtained for the two sided chain because they only depend on the distribution of $(X_n)_{n=1}^N$. 

The extension to the two sided case is done in two steps. Firstly, let us take a probability space $(\Om_1,\cF_1,\bbP_1)$ on which the  original chain $(X_n)_{n\in\bbN}$ is defined. Secondly, let us take a probability space $(\Om_2,\cF_2,\bbP_2)$ which supports an iid sequence of $\cX_1$-valued random variables which have the same distribution as $X_1$. Let us denote such a sequence by $(X_k)_{k=-\infty}^{0}$.  Finally,  consider the probability space $(\Om,\cF,\bbP):=(\Om_2\times\Om_1,\cF_2\times\cF_1,\bbP_2\times\bbP_1)$. Let us define random variables $\tilde X_n:(\Om,\cF,\bbP)\to\bbR$ by setting $\tilde X_n(\om_2,\om_1)=X_n(\om_1), n\geq 1$ and $\tilde X_n(\om_2,\om_1)=X_n(\om_2), n\leq 0$, where $\om_i\in\Om_i, i=1,2$. Then the new two sided sequence $(\tilde X_n)_{n\in\bbZ}$ forms a uniformly elliptic Markov chain. 
Notice also that \eqref{Doeblin} still holds true for the new chain $(\tilde X_n)_{n\in\bbZ}$, and with the same same ellipticity constant $C$. Indeed, for $n\geq 1$ it holds since \eqref{Doeblin} holds for $(X_n)_{n\geq1}$ and for $n\leq 0$ it actually holds with $C=1$ since $\tilde X_{n+1}$ and $\tilde X_{n}$ are independent.

Let us now prove (i).
It is clearly enough to prove it in the case when $\bbE(Y_n)=0$.
Now,  for every positive integers $n,k$ and $l$  we have
\begin{equation}\label{Bridge}
\bbP(X_{n}\in A|X_{n-l}=a,X_{n+k}=b)=\int_{A} p_n^{(n-l,n+k)}(y|a,b)d\mu_n(y)
\end{equation}
with 
$$
 p_n^{(n-l,n+k)}(y|a,b)=\frac{p_{n-l}^{(l)}(a,y)p_{n}^{(k)}(y,b)}{p_{n-l}^{(l+k)}(a,b)}
$$
and  $p_{m}^{(s)}$ is the transition density of $X_{m+s}$ given $X_m$.
Hence, by Lemma \ref{dens lemma}, for every possible value of $Y_n=f_n(X_n)$ we have
\begin{equation}\label{UP}
C^{-3}\bbP(Y_n=x)\leq \bbP(Y_n=x|X_{n-l},X_{n+k})\leq C^3\bbP(Y_n=x).
\end{equation}
Note that when $n$ is smaller than the smallest index $n_1$ in $\cB$, then we only  need to condition on $X_{n_1}$, 
but in this case we still derive that
$$
C^{-3}\bbP(Y_n=x)\leq \bbP(Y_n=x|X_{n_1})\leq C^3\bbP(Y_n=x)
$$
from \eqref{Bridge}, by further conditioning on $X_{0}$.
We  conclude that
$$
\text{Var}(Y_n|\cE)=
\sum_{x}\bbP(Y_n=x|\cE)\left(x-\bbE[Y_n|\cE]\right)^2
$$
$$
\geq C^{-3}\sum_{x}\bbP(Y_n=x)\left(x-\bbE[Y_n|\cE]\right)^2
\geq C^{-3}\text{Var}(Y_n)
$$
since $\bbE(Y_n-a)^2\!\!\geq\!\! \text{Var}(Y_n)$ for every $a\in\bbR$. On the other hand, using again \eqref{UP} we get
\[
\text{Var}(Y_n|\cE)\leq \bbE[Y_n^2|\cE]=\sum_{x}\bbP(Y_n=x|\cE)x^2\leq C^3\sum_{x}\bbP(Y_n=x)x^2=C^3\text{Var}(Y_n)
\]
and   \eqref{VarRel} follows.
\vskip3mm

 To prove (ii),
we use  the fact that if $Z$ is an integer valued random variable with $\|Z\|_{L^\infty}\leq K$
then
\begin{equation}
\label{Q-V} \frac{q(Z)}4\leq \text{Var}(Z)\leq 8 K^3 q(Z)
\end{equation}
where $q(Z)$ is the probability that $Z$ takes its second most likely value.
Indeed let $Z'$ and $Z''$ be independent copies of
$Z$, and let
$\bar i$ and $\hat i$ be the most likely and the second most likely 
values of $Z$. Then
$$ \text{Var}(Z)=\frac12\bbE[(Z'-Z'')^2]=$$$$\frac{1}{2} \sum_{|i|, |j|\leq K} \bbP(Z=i)\bbP(Z=j) (i-j)^2\geq
\frac{1}{2} \bbP(Z=\hat i) \bbP(Z\neq \hat i)=\frac{q(Z)(1-q(Z))}{2}\geq \frac{q(Z)}4  $$
since $q(Z)\leq \frac{1}{2}.$ On the other hand, using the above formula for 
$\text{Var}(Z)$ we get
$$  \text{Var}(Z)\leq \frac{1}{2} \times (2K)^2 \times \bbP(Z'\neq \bar i \text{ or } Z''\neq \bar i)
\leq (2K)^2 \times \bbP(Z\neq \bar i)\leq (2K)^2 \times 2K q(Z).
$$ 
Applying \eqref{Q-V} with $Z=Y_n$ mod $m$ and using (i) proves (ii).
\end{proof}

\subsection{Conditional chains}
\begin{lemma}\label{CondChain}
Let $\cE=\cE_N$ be like in Lemma \ref{FiberQuantLemma}.
After conditioning on $\cE$, for almost every realization of $\{X_n: n\in\cB\}$ the sequence $\{X_n: n\not\in\cE\}$ forms a uniformly elliptic Markov chain. 
More precisely, 
let us write $\cB=\{n_1<n_2<....<n_d\}$, where both $n_i$ and $d$ might also depend on $N$. For the sake of convenience, let us also set $n_0=0$ and $n_{d+1}=\infty$. Then, for almost every realization we have the following:
\vskip0.2cm
(i)  
The random variables $\textbf{X}_0,\textbf{X}_1,...,\textbf{X}_{d}$, where 
$\textbf{X}_s=\{X_n: n_s<n<n_{s+1}\}$ are conditionally independent, namely they are independent with respect to $\bbP_\cE$,  where $\bbP_{\cE}(\cdot)$ denotes $\bbP(\cdot|\cE)$. 
\vskip0.2cm

(ii) If $n$ and $n+1$ belong to the same block $(n_s,n_{s+1})$ then  
$$
\bbP_\cE(X_{n+1}\in A|X_n=x)=\int_{A} p_{n,\cE}(x,y)d\mu_{n+1,\cE}(y)
$$
with  a density satisfying $C^{-6}\leq p_{\cE,n}(x,y)\leq C^6$,
where $d\mu_{n+1,\cE}$ denotes the law of $X_{n+1}$ given $\cE$.
\end{lemma}
\begin{proof}
The first part follows because $\{X_n\}$ is a Markov chain, and it does not require ellipticity.

To prove the second part, notice that by \eqref{Bridge}, together with 
 Lemma \ref{dens lemma}, the conditional law of $X_n$ is equivalent to the law of $X_n$, and the Radon-Nikodym derivative is bounded above by $C^{3}$ and  below by $C^{-3}$. Note that if $n<n_1$ then we can still use \eqref{Bridge} by taking $n-l=0$ and setting $X_0$ to be an independent copy of $X_1$ which is independent of $\{X_n\}$.
 
Next,  since $\{X_n: n_s<n<n_{s+1}\}$ are conditionally independent,  it is enough to show that each $\{X_n: n_s<n<n_{s+1}\}$ forms a  uniformly elliptic Markov chain after conditioning on $\cE$. To show that, let $n$ satisfy that $n_{s}<n<n+1<n_{s+1}$ for some $s$. If $0<n_s$ and $n_{s+1}<\infty$ then
$$
\bbP_{\cE}(X_{n+1}\in A|X_n)=\bbP(X_{n+1}\in A|X_{n},X_{n_{s+1}})
$$
and so,
it follows from \eqref{Bridge} that 
$$
\bbP_{\cE}(X_{n+1}\in A|X_n)=\int_{A} p_n^{(n_s,n_{s+1})}(y|X_{n},X_{n_{s+1}})d\mu_{n+1}(y)
$$
where the densities are bounded above by $C^3$ and below by $C^{-3}$. Now the result follows since $\mu_{n+1}$ and $\mu_{n+1,\cE}$ are equivalent with Radon-Nikodym derivatives bounded between $C^{-3}$ and  and $C^{3}$. The proof when $n_s=1$ is similar, and the case $n_{s+1}=\infty$ reduces to the unconditioned chain.
\end{proof}

\section{Edgeworth expansions via Fourier analysis}
The first step of the proof of all our results is as follows.
By the formula of the $k$-th Fourier coefficient of the 
$2\pi$-periodic function $\varphi_N(t)=\bbE(e^{it S_N})$ we see that 
$$
2\pi \bbP(S_N=k)=\int_{-\pi}^{\pi}e^{-itk}\bbE(e^{it S_N})dt.
$$
Next,
Let $\cR=\{t_j\}$ be the set of all $K$--resonant points  (see Definition \ref{DefResP}).
Divide $[-\pi,\pi]$ into intervals $I_j$ of small size $2\delta$
such that each interval contains at most one resonant point which is strictly inside
$I_j.$ We call an interval resonant if it contains a resonant point inside. Then
\begin{equation}\label{SplitInt}
2\pi \bbP(S_N=k)=\sum_{j}\int_{I_j} e^{-itk}\bbE(e^{it S_N})dt.
\end{equation}
 In the case of sums of independent identically distributed
integer valued random variables, the Edgeworth expansion comes from the expansion 
of the characteristic function near zero while the other intervals give negligible 
contributions. In this section we obtain a similar estimates for  integer valued additive
functionals of uniformly elliptic Markov chains. However, in contrast to the iid case,
in order to be able to disregard the contribution of an interval $I_j$ we need to assume that
this interval is either non-resonant, or it is resonant but the value of $M_N(m)$ 
(where $m$ is the denominator of the corresponding resonant point)
is large.

\subsection{The contribution a neighborhood of $0$}
In this section we will estimate the integral $\int_{I_j} e^{-itk}\bbE(e^{it S_N})dt$ when $t_j=0\in I_j$. Namely, we will expand the integral 
$$
\int_{-\del}^\del \bbE(e^{it S_N})dt=\sig_N^{-1}\int_{-\del\sig_N}^{\del \sig_N}e^{it\bbE[S_N]/\sig_N}\bbE(e^{it(S_N-\bbE[S_N])/\sig_N})dt
$$
for a sufficiently small $\del=\del_r$. First, by Lemma \ref{L}, if $\del$ is small enough then 
$$
\int_{-\del\sig_N}^{\del \sig_N}e^{it\bbE[S_N]}\bbE(e^{it(S_N-\bbE[S_N])/\sig_N})dt=
\int_{-\infty}^{\infty}e^{it\bbE[S_N]/\sig_N}e^{-t^2/2}\left(1+Q_{r,N}(t)\right)dt+O(\sig_N^{-r-1}).
$$
Second, recall that  for every real $\al$ we have
$$
\int_{-\infty}^{\infty} e^{-i\al h}e^{-h^2/2}h^k dh=\sqrt{2\pi}(-1)^kH_k(\al)\fg(\al),
$$
where $H_k$ is the $k$-th Hermite polynomial and $\fg(u)=\frac{1}{\sqrt{2\pi}} e^{-u^2/2}$. Applying this formula with $\al=-\bbE[S_N]/\sig_N$ and taking into account \eqref{Q} and \eqref{P} 
we get the following result.
\begin{proposition}\label{0prop}
There are  polynomials $P_{0,b,N}$ with uniformly bounded coefficients and degrees not depending on $N$ so that for $t_j=0$, for every $r\geq1$ we have the following: if the length of the resonant interval $I_j$ around $0$ is small enough then 
$$
\int_{I_j} e^{-itk}\bbE(e^{it S_N})dt=\sum_{b=1}^r \frac{P_{a, b, N} ((k-a_N)/\sigma_N)}{\sigma_N^b}
\fg((k-a_N)/\sigma_N)+o(\sigma_N^{-r}).
$$
\end{proposition}
The above Proposition shows that the contribution of the neighborhood of $0$ corresponds to the polynomials $P_{0,b,N}$ and the choice $a=0$ (in the notations of Theorem~\ref{MainThm}).

\subsection{The negligible contribution: non-resonant intervals and resonant points with  $M_N(m)\geq R\ln \sig_N$ and the proof of Theorem \ref{MainThm} (ii)}\label{Sec4}


In this section we will prove the following result.

\begin{proposition}\label{LLL1}
Let $r\in\bbN$. There exist constants $\del,c>0$ which depend only on $K$ and $r$ and constants $R=R(K,C,r)$ which depend only on $K$, $r$ and $C$ (from \eqref{Doeblin}) 
with the following properties. 
 Suppose that either $t$ is not a $K$-resonant point or $t=2\pi \ell/m$ is a non-zero $K$ resonant point such that $M_N(m)\geq R\ln \sig_N $.  If the size of $I_j$ does not exceed $2\del$ then
 $$
\sup_k\left|\int_{I_j}e^{-ixk}\mathbb E[e^{ix S_N}]dx\right|=o(\sig_N^{-r}).
$$
\end{proposition}

\begin{proof}
First,  since the random variables $Y_n$ are uniformly bounded,  we will use a somehow simpler way to describe the conditions on $M_N$.  Let $q_n(m)$ denote the second largest value among $P(Y_n\equiv j\text{ mod }m)$, 
$j=0,1,\dots, m-1$.  Set
\[
\widetilde M_N(m)=\sum_{n=1}^N q_n(m)
\quad
\text{and} 
\quad
\widetilde M_N=\min_{2\leq m\leq 2K}\tilde M_N(m).
\]
Then there are constants $c_1,c_2>0$ such that for every $N$ and $2\leq m\leq 2K$ we have 
\begin{equation}\label{M remark}
c_1\widetilde M_N(m)\leq M_N(m)\leq c_2\widetilde M_N(m).
\end{equation}
Consequently, 
$$
c_1\widetilde M_N\leq M_N\leq c_2\widetilde M_N.
$$
Thus, up to changing the constant $R$, we can  replace $M_N$ with $\widetilde M_N$.

Let $t=2\pi \ell/m$  be a non-zero $K$--resonant point.
By \ref{M remark}, to prove the proposition for such points it is enough to prove that the contribution is negligible if 
$\DS \tilde M_N(m)=\sum_{n=1}^{N}q_n(m)\geq R\ln \sig_N$ for some $R$ large enough (which will be determined soon). For the sake of simplicity, 
assume that $\DS \sum_{n=1}^{[N/2]}q_{2n}(m)\geq \frac{R\ln\sig_N}2$  (otherwise we will work with odd indexes instead). Let us condition on  $X_1,X_3,X_5,\dots$. Then 
$X_2,X_4,X_6,\dots$  are independent after such a conditioning. Moreover,  by 
Lemma~\ref{FiberQuantLemma} (ii)
the ratio between $q_{2n}(m)$  and their conditional versions is uniformly bounded and bounded away from $0$. Therefore, we can assume that, after the conditioning, we  have 
$\DS Q_N^{even}:=\sum_{n=1}^{[n/2]}q_{2n}(m|X_{1},X_3,...)\geq (A^{-1}R/2)\ln\sig_N$ 
 where $A$ comes from Lemma \ref{FiberQuantLemma}(ii).
Applying \cite[Lemma 3.4]{DH} with the conditional law we see that there exist $\del,C_0>0$ that depend only on $K$ such that 
$$
\sup_{|x-t|\leq\del}\left|\bbE[e^{ix S_N}|X_1,X_3,...]\right|\leq 
e^{-c_0  Q_N^{even}} \leq e^{-\frac12c_0A^{-1}R\ln\sig_N}=\sig_N^{-\frac12 c_0A^{-1}R},\,\text{a.s.}
$$
 We note that the fact that $\del$ and $c_0$ depend only on $K$ and $r$ was not explicitly formulated in \cite[Lemma 3.4]{DH}, but this can be verified taking a closer look at its (short) proof. 
We conclude that  for every real $x\in[t-\del,t+\del]$,
$$
\left|E[e^{ix S_N}]\right|=\left|\bbE[\bbE[e^{ix S_N}|X_1,X_3,...]]\right|\leq 
\left|\bbE[|\bbE[e^{ix S_N}|X_1,X_3,...]|]\right|\leq 
\sig_N^{-\frac12c_0A^{-1}R}=o(\sig_N^{-r})
$$
where the last estimate holds true if $R>2rA/c_0$. As a consequence, if $t=t_j$ for some $j$ and the length of $I_j$ does not exceed $2\del$ then 
$$
\sup_k\left|\int_{I_j}e^{-ixk}E[e^{ix S_N}]dx\right|\leq 2\del
e^{-\frac12c_0A^{-1}R\sig_N}=o(\sig_N^{-r}).
$$

Next, let us assume that  $[t-\del, t+\del]$ is a non resonant interval. 
In this case, we will prove the proposition in the case when the sums of the variances of $f_{2n}(X_{2n}),n\leq N/2$ is larger than the sum of corresponding sum along the odd indexes (the proof when the sum along the odd indexes is larger requires minor notational modifications). 
After conditioning on $X_1,X_3,...$ and then,
by repeating the arguments in the independent case in
\cite[Lemma 3.3]{HH}, we see that  there are constants $\del,c>0$ which depend only on $K$ such that for every $x\in[t-\del,t+\del]$ we have  (a.s.)
$$
\left|\bbE[e^{ix S_N}|X_1,X_3,...]\right|\leq e^{-c\sum_{n=1}^{[N/2]}\text{Var}(Y_{2n}|X_1,X_3....)}.
$$
By 
Lemma~\ref{FiberQuantLemma} (ii) the conditional variances are proportional to the unconditioned variances. We conclude that there are constants $c_0,c_1>0$ such that (a.s.),
$$
\sup_{x\in[t-\del,t+\del]}\left|\bbE[e^{ix S_N}|X_1,X_3,...]\right|\leq e^{-2c_0\sum_{n=1}^{N/2}\text{Var}(Y_{2n})}\leq 
e^{-c_0\sum_{n=1}^{N}\text{Var}(Y_{n})}\leq
e^{-c_1\text{Var}(S_N)}
$$
where in the last inequality we have used Lemma \ref{MomEst}(i). 
Integrating both sides  and taking $I_j$ with length smaller than $2\del$, we conclude that 
$$
\sup_k\left|\int_{I_j}e^{-ixk}\mathbb E[e^{ix S_N}]dx\right|\leq 2\del
e^{-c_1\sig_N^2}=o(\sig_N^{-r})
$$
for all $r$.
\end{proof}

\begin{proof}[Proof of the second part of Theorem \ref{MainThm}]
 
To show that the condition $M_N(m)\geq R\ln \sig_N$ is sufficient for the regular expansions, 
consider intervals  $I_j$ in the decomposition \eqref{SplitInt} of sufficiently small size such that the estimate in Propositions \ref{0prop} and \ref{LLL1} hold true. 
In the circumstances of Theorem \ref{MainThm} (ii), the conditions of Proposition \ref{LLL1} hold true for every $m$ (for some $R$). Now Theorem \ref{MainThm} (ii)
follows by combining \eqref{SplitInt}, Proposition \ref{0prop} and Proposition \ref{LLL1}.
\end{proof}

\begin{remark}
Note that in the circumstances of Theorem \ref{MainThm} (ii) the  conditional Edgeworth expansion of order $r$ follows from  the unconditioned one. 
Indeed, by Lemma~\ref{CondChain},  after conditioning on a finite number of variables we still  get a uniformly elliptic Markov chain. Moreover, by Lemma \ref{FiberQuantLemma}(ii) $q_n(m)$ can only change by a multiplicative constant. Hence, if $M_N(m)\geq R\ln \sig_N$
for large enough $R$ and all denominators $m$ of nonzero $K$--resonant points, then the same remains true after such a conditioning.
\end{remark}

\section{Contribution of nonzero resonant intervals with $M_N(m)\leq R\ln\sig_N$ and the proof of Theorem \ref{MainThm}.}\label{Sec 5}
In this section we prove Theorem \ref{MainThm}.
Let $t_j=2\pi l/m$ be a nonzero resonant point such that 
$$
\tilde M_N(m)=\sum_{n=1}^N q_n(m)\leq R\ln \sigma_N,\,\, \sigma_N=\sqrt{\text{Var}(S_N)}
$$
(recall that by  \ref{M remark}\, $\tilde M_N(m)$ and $M_N(m)$ are proportional). 
Let us fix some $\bar\eps>0$, and let $N_0=N_0(t_j,N)$ be the number of indexes $n$ between $1$ to $N$ so that $q_n(m)\geq \bar\epsilon$. Then $N_0\leq \frac{R\ln\left(\sigma_N\right)}{\bar\epsilon}$. Let us denote the latter indexes by $n_1<n_2<\dots<n_{N_0}$ and set $\cB=\cB_N=\{n_1,...,n_{N_0}\}$. 
Let $\cE=\cE_N$ be the $\sig$-algebra generated by $\{X_n: n\in\cB\}$. 

\subsection{More on conditioning}\label{Sec 5.1}

\begin{lemma}\label{L1}
There is a constant $C>0$ such that
\begin{equation}\label{ExpDiff}
\ess\sup|\bbE(S_N|\cE)-\bbE(S_N)|\leq CN_0\leq \frac{CR\ln\left(\Var(S_N)\right)}{\bar\eps}
\end{equation}
and
\begin{equation}\label{VarDiff}
\ess\sup |\Var(S_N|\cE)-\Var(S_N)|\leq C\ln^2\left(\Var(S_N)\right).
\end{equation}
\end{lemma}

Before proving the lemma we will need the following auxiliary result.
Given a finite set $B\subset \bbN$, we denote
$\DS S_B=\sum_{n\in B}Y_n$.

\begin{lemma}\label{L1 aux}
 Let $0<A_1<A_2$. If $A_1$ and $A_2-A_1$ are sufficiently large then  there is a deterministic partition of $\{1,...,N\}$ into blocks $B_1,B_2,...,B_{a_N}$ such that: 
\vskip0.2cm
(i) $a_N\asymp \sigma_N^2$; 
\vskip0.2cm
(ii) for all $k$, both the numbers $\Var(S_{B_k})$ and the random variables 
$\Var(S_{B_k}|\cE)$ lie between $A_1$ and $A_2$.
\end{lemma}

\begin{proof}
 The existence of a partition $B_i$ satisfying part (i) and the estimates on the unconditioned variance in  part (ii) of Lemma \ref{L1} follow from the arguments in the proof of \cite[Corollary 3.6]{DH}. 

In order to pass to the conditional variance of each block, recall that the conditioned chain is uniformly elliptic with deterministic ellipticity constant $C^6$ (see Lemma \ref{CondChain}). Thus, we can apply Lemma \ref{MomEst} after the conditioning in order to deduce that 
$$
C_1'\sum_{j\in B_k}\text{Var}(Y_j|\cE)\leq \text{Var}(S_{B_k}|\cE)\leq C_2'\sum_{j\in B_k}\text{Var}(Y_j|\cE)
$$
for some constants $C_1',C_2'>0$.
Applying \eqref{VarRel} to the conditioned chain  we see that 
$$
D_1\sum_{j\in B_k}\text{Var}(Y_j)\leq \sum_{j\in B_k}\text{Var}(Y_j|\cE)\leq D_2\sum_{j\in B_k}\text{Var}(Y_j)
$$
for some constants $D_1,D_2>0$. 
Finally,  applying Lemma \ref{MomEst} (i) with the unconditioned chain we obtain that 
$
\DS \sum_{n\in B_k}\Var(Y_n)\asymp \Var(S_{B_k})
$
and hence there are constants $D_1',D_2'>0$ such that, a.s. we have 
$\DS
D'_1\Var(S_{B_k})\leq {\Var}(S_{B_k}|\cE)\leq D'_2\Var(S_{B_k}).
$
\end{proof}

\begin{proof}[Proof of Lemma \ref{L1}]
In order to prove (\ref{ExpDiff}), 
 consider an integer $n$ such that $n_{i}\!\!<\!\!n\!\!<\!\!n_{i+1}$ for some $i$. Then 
\begin{equation}\label{E}
\bbE(Y_n|\cE)-\bbE(Y_n)=\int\left(p_n^{(n_i,n_{i+1})}(x|X_{n_i},X_{n_{i+1}})-1\right)f_n(x)  d\mu_n(x) 
\end{equation}
where $p_n^{(n_i,n_{i+1})}(x|X_{n_i},X_{n_{i+1}})$ is defined by 
$$
p_n^{(n_i,n_{i+1})}(x|X_{n_i},X_{n_{i+1}})=\frac{p_{n_i}^{(n-n_i)}(X_{n_i},x)p_{n}^{(n_{i+1}-n)}(x,X_{n_{i+1}})}{p_{n_{i}}^{(n_{i+1}-n_i)}(X_{n_i},X_{n_{i+1}})}.
$$

By Lemmas \ref{dens lemma} and \ref{psi dens lem} we have
$$
\left|p_{n_{i}}^{(n-n_i)}(X_{n_i},x)-1\right|\leq C_1\del^{n-n_i}, \,\,\,
\left|p_{n_{i}}^{(n_{i+1}-n)}(x,X_{n_{i+1}})-1\right|\leq C_1\del^{n_{i+1}-n},
$$
$$
C^{-1}\leq p_{n_{i-1}}^{(n_{i+1}-n_i)}(X_{n_i},X_{n_{i+1}})\leq C,\,\,\,
\left| p_{n_{i-1}}^{(n_{i+1}-n_i)}(X_{n_i},X_{n_{i+1}})-1\right|\leq C_1\del^{n_{i+1}-n_i}.
$$
We thus conclude that 
$\DS
\left|p_n^{(n_i,n_{i+1})}(x|X_{n_i},X_{n_{i+1}})-1\right|\leq C_2\del^{\min(n-n_i, n_{i+1}-n)}
$
for some constant $C_2$.
Therefore,
\begin{equation}\label{Expec1}
\left|\bbE(Y_n|\cE)-\bbE(Y_n)\right|\leq C_2\bbE[|Y_n|]\del^{\min(n-n_i, n_{i+1}-n)}.
\end{equation}
Similarly when $n<n_1$ or $n>n_{N_0}$ we have 
\[
\left|\bbE(Y_n|\cE)-\bbE(Y_n)\right|\leq C\bbE[|Y_n|]\del^{d(n,\cB_N)}
\]
where $\cB_N=\{n_1,n_2,...,n_{N_0}\}$ and $d(n,B)=\min\{|n-b|: b\in B\}$ 
and we set $n_0=0$ and $n_{N_0+1}=N$. It follows that there exists a constant $C>0$ so that for all $i$,
\begin{equation}\label{Expectation}
\sum_{n_i<n\leq n_{i+1}}\left|\bbE(Y_n|\cE)-\bbE(Y_n)\right|\leq C.
\end{equation}
Therefore, 
\begin{equation}\label{Therefore}
|\bbE(S_N|\cE)-\bbE(S_N)|\leq \sum_{i=0}^{N_0}\sum_{n_i<n\leq n_{i+1}}\left|\bbE(Y_n|\cE)-\bbE(Y_n)\right|\leq C(N_0+1).
\end{equation}

\smallskip
Now we prove (\ref{VarDiff}). Let $A_1<A_2$ be like in Lemma \ref{L1 aux}, and let $B_1,...,B_{a_N}$ be the partition specified there.
Let us define $Z_k=S_{B_k}$.  Let $\del\in(0,1)$ be like in \eqref{psi} and let $A$ be a sufficiently large constant so that 
\begin{equation}\label{A cond}
a_{N}^2\del^{A\ln a_N}\leq 1.
\end{equation}
Let  $\cN_N$
be the set of indexes $1\leq k\leq a_N$ such that the distance between $B_k$ and $\cB_N$ is at least $A\ln a_N$. Set 
$$
\tilde S_N=\sum_{k\in\cN_N}Z_k.
$$
We  first claim that 
\begin{equation}\label{V first}
\text{Var}(S_N)=\text{Var}(\tilde S_N)+O(\ln^2  a_N) \,\text{ and }\,\|\text{Var}(S_N|\cE)-\text{Var}(\tilde S_N|\cE)\|_{L^\infty}=O(\ln^2 a_N).
\end{equation}
In order to prove \eqref{V first}, we have
$$
\text{Cov}(Z_{n+k},Z_{n})=\bbE\left[(Z_{n}-\bbE[Z_n])(\bbE[Z_{n+k}-\bbE[Z_{n+k}]|Z_n)\right]
$$
and so by \eqref{psi} and \eqref{PsiRep},
\begin{equation}\label{ee}
   |\text{Cov}(Z_{n+k},Z_{n})|\leq C_1\|Z_n-\bbE[Z_n]\|_{L^1}\del^k\leq C\del^k 
\end{equation}
for some constant $C>0$, where we have used that $\text{Var}(Z_n)=\|Z_n-\bbE[Z_n]\|_{L^2}^2\leq A_2$.  
Since the conditioned Markov chain is uniformly elliptic with deterministic ellipticity constant, \eqref{psi} also holds true after conditioning on $\cE$, with possibly a different deterministic constant $C_1$. Since $\text{Var}(Z_n|\cE)\leq A_2$, we obtain similarly that  
\begin{equation}\label{ee1}
|\text{Cov}(Z_{n+k},Z_{n}|\cE)|\leq C\del^k,\,\text{a.s}
\end{equation}
where $\text{Cov}(X,Y|\cE)$ denotes the conditional covariance of two random variables $X$ and $Y$.
Next, we claim that there is  a constant $C'>0$ such that
 for every $\Gamma\subset\{1,....,a_N\}$ and  
 $k\not\in\Gamma$ we have
\begin{equation}
\label{RemoveOne}
|\text{Var}(\tilde S_\Gamma)-\text{Var}(\tilde S_\Gamma-Z_k)|\leq C'\,\text{ and }\,|\text{Var}(S_N|\cE)-\text{Var}(S_N-Z_k|\cE)|\leq C',
\end{equation}
where $\tilde S_\Gamma=\sum_{j\in\Gamma}Z_j$.
Indeed we write 
$$
\text{Var}(\tilde S_\Gamma)=\text{Var}(\tilde S_\Gamma-Z_k+Z_k)=\text{Var}(\tilde S_\Gamma-Z_k)+\text{Var}(Z_k)+2\text{Cov}(\tilde S_\Gamma,Z_k).
$$
Now by \eqref{ee},
$$
\left|\text{Cov}(\tilde S_\Gamma,Z_k)\right|\leq \sum_{j<k}|\text{Cov}(Z_j,Z_k)|+\sum_{j>k}|\text{Cov}(Z_j,Z_k)|\leq 2C/\del
$$
and hence $|\text{Var}(\tilde S_\Gamma)-\text{Var}(\tilde S_\Gamma-Z_k)|\leq A_2+2C/\del$.
Similarly, using \eqref{ee1} we see that $|\text{Var}(\tilde S_\Gamma|\cE)-\text{Var}(\tilde S_\Gamma-Z_k|\cE)|\leq A_2+2C/\del$ (a.s.).
Now (\ref{V first}) follows by a  repeated application  of \eqref{RemoveOne},
 taking into account that $\tilde S_N$ was obtained by removing at most $O(\ln^2 a_N)$ individual summands $Z_j$.

Next, let us show that there exists $C_1>0$ so that
\begin{equation}\label{VarDiff-Temp}
|\text{Var}(\tilde S_N|\cE)-\text{Var}(\tilde S_N)|\leq C_1,\,\, \text{ a.s.}
\end{equation}
Note that in order to prove \eqref{VarDiff-Temp} we can  assume that $\bbE[Z_k]=0$ for all $k$ (otherwise just replace $Z_k$ by $Z_k-\bbE[Z_k]$).
Take
$A$ satisfying \eqref{A cond}. Let $1\leq k_1\leq k_2\leq a_N$ be so that $d(B_{k_i},\cB_N)\geq A\ln a_N$ for $i=1,2$, where $d(\Gamma,\Delta)$ denotes the distance between two finite sets $\Gamma$ and $\Delta$.  Then there are $s_1,s_2$ so that $B_{k_i}$ is contained in $(n_{s_i},n_{s_i+1})$,  for $i=1,2.$
Assume first that $s_1\not=s_2$. Then $Z_{k_1}$ and $Z_{k_2}$ are independent given $\cE$ and 
$$
\max\left(|\text{Cov}(Z_{k_1},Z_{k_2})|,(|\text{Cov}(Z_{k_1},Z_{k_2}|\cE)|\right) \leq 
C\del^{A \ln a_N}.
$$
Therefore, by \eqref{A cond}, the contribution of such pairs to the left hand side of (\ref{VarDiff-Temp}) is $O(1)$.

Next, let us assume that $B_{k_1},B_{k_2}\subset(n_{s},n_{s+1})$ for some $s$.
We will first show that, with  $W_{k_i}=\{f_n(X_n):\,n\in B_{k_i}\}$ \,$i=1,2$, for all possible values $u=(u_n)_{n\in B_{k_1}}$ and $v=(v_n)_{n\in B_{k_2}}$ of $W_{k_1}$ and $W_{k_2}$, respectively,  
 and all possible values $x$ and $y$ of $X_{n_s}$ and $X_{n_{s+1}}$ respectively,
we have
\begin{eqnarray}\label{Prob Diff}
P(W_{k_1}=v, W_{k_2}=u| X_{n_s}=x, X_{n_{s+1}}=y)-P(W_{k_1}=v, W_{k_2}=u)\\
=P(W_{k_1}=v, W_{k_2}=u)
O(\del^{A\ln a_N}).\nonumber
\end{eqnarray}
Under \eqref{Prob Diff} we have
$$
\left|\bbE[ Z_{k_1} Z_{k_2}]-\bbE[ Z_{k_1} Z_{k_2}|\cE]\right|\leq \sum_{z_{k_1},z_{k_2}}|z_{k_1}z_{k_2}|\left|\bbP(Z_{k_i}=z_{k_i})-\bbP(Z_{k_i}=z_{k_i}|\cE)\right|\leq 
$$
$$
C\del^{A\ln a_N}\sum_{z_{k_1},z_{k_2}}|z_{k_1}z_{k_2}|\bbP(Z_{k_i}=z_{k_i})=C\del^{A\ln a_N}\bbE[|Z_{k_1}Z_{k_2}|].
$$
Here $\bbP(Z_{k_i}=z_{k_i})=\bbP(Z_{k_i}=z_{k_i};i=1,2)$,  and 
$\bbP(Z_{k_i}=z_{k_i}|\cE)$ is defined
 similarly.
Similarly, for $i=1,2$ we have
$$
\left|\bbE[Z_{k_i}]-\bbE[Z_{k_i}|\cE]\right|\leq C\del^{A\ln a_N}\bbE[|Z_{k_i}|].
$$
Taking into account  our assumption that $\bbE[Z_k]=0$ and that $A_1\leq \text{Var}(Z_k)\leq A_2$, we conclude that

\begin{equation}
\label{Cov-CondCov}
\left|\text{Cov}(Z_{k_1},Z_{k_2})-\text{Cov}(Z_{k_1},Z_{k_2}|\cE)\right|\leq 
C\sqrt{\text{Var}(Z_{k_1})\text{Var}(Z_{k_2})}\del^{A\ln a_N}=O(a_N^{-2})
\end{equation}
which yields that the contribution of such pairs $k_1$ and $k_2$ to the left hand side of \eqref{VarDiff-Temp} is also $O(1)$. 
This completes the proof of \eqref{VarDiff-Temp}.

Now let us prove \eqref{Prob Diff}. Let $X(B_{k_i})=\{X_n:\,n\in B_{k_i}\}$ and let 
$$
\Gamma_1=\left\{X(B_{k_1}): 
 f_j(X_j)=u_n,\,\,\, \forall n\in B_{k_1}\right\},  \quad \Gamma_2=\left\{X(B_{k_2}):  f_n(X_n)=v_n, \,\,\forall n\in B_{k_2}\right\}.
$$
For every $n$, $k$ and $x=(x_0,...,x_{k})$ let us write
\[
p^{n,n+k}(x)=\prod_{j=1}^{k}p_{n+j-1}(x_{j-1},x_j)
\]
where we recall that $y\to p_{n}(x,y)$ is the transition density of $X_{n+1}$ given $X_n=x$. 
Denote
$$
\bar x=(x_j)_{j=n_s+1}^{n_{s+1}}, \quad
\bar x^{(i)}=(x_j)_{j\in B_{k_i}},\,\,i=1,2
$$
and let 
$$
D=\{\bar x: \bar x^{(i)}\in\Gamma_i, i=1,2\}.
$$
Let us assume that $k_1\leq k_2$ and write $B_{k_i}=[\al_i,\be_i]\cap\bbN$. 
Then the difference on the left hand side of (\ref{Prob Diff}) can be written as
$$
\int_{D}p^{\al_1,\beta_1}(\bar x^{(1)})p_{\beta_1}^{(\al_2-\beta_1)}(x_{\al_2})p^{\al_2,\beta_2}(\bar x^{(2)})\left(
\frac{p_{n_s}^{(\al_1-n_s)}(x,x_{\al_1})p_{\beta_2}^{(n_{s+1}-\beta_2)}(x_{\beta_2},y)}{p_{n_s}^{(n_{s+1}-n_s)}(x,y)}-1\right)\prod_{t=\al_1}^{\beta_2}d\mu_{t}(x_{t})
$$
where we recall that $y\to p_{j}^{(k)}(x,y)$ is the transition density of $X_{j+k}$ given $X_{j}=x$.
Using Lemma \ref{psi dens lem}, the term in the parenthesis can be  bounded exactly as the corresponding term from (\ref{E}). 
That is, arguing like at the beginning of the proof of Lemma \ref{L1}, we get 
 $$
 \left|
\frac{p_{n_s}^{(\al_1-n_s)}(x,x_{\al_1})p_{\beta_2}^{(n_{s+1}-\beta_2)}(x_{\beta_2},y)}{p_{n_s}^{(n_{s+1}-n_s)}(x,y)}-1\right|\leq C_2\del^{\min(\al_1-n_s,n_{s+1}-\beta_2)}\leq C_2 \del^{A\ln a_N}.
$$
The second inequality follows since 
 since  $d(B_{k_i},\cB)\geq A\ln a_N$ and  $B_{k_i}\subset (n_s,n_{s+1})$,
 which implies that 
 $n_{s+1}-n_{s}\geq A\ln a_N$, and the proof of Lemma \ref{L1} is completed.
\end{proof}

\subsection{The conditional cumulant generating function around nonzero resonant points and its additive behavior}
First, by Lemma  \cite[Lemma 3.4]{DS} we can find a bounded sequences of integers $(c_n)$ so that 
$\DS \bbE(S_N)-\sum_{n=1}^Nc_n$ is bounded in $N$. 
Therefore, by replacing $Y_n$ with $Y_n-c_n$ we can assume that 
$\DS \sup_N|\bbE(S_N)|<\infty$. Consequently,  
\begin{equation}\label{E bound}
  \sup_{a<b}|\bbE[S_b-S_a]|<\infty.  
\end{equation}

Next, let $\cE=\cE_N=\sig\{X_j: j\in \cB_N\}$ be like at the beginning of \S \ref{Sec4}. Denote $X(\cB_N)=(X_j)_{j\in\cB_N}$.
Fix a nonzero $K$--resonant point $t=t_j=2\pi l/m$. Let  $j(Y_n,m|\cE)$ be the most likely residue of 
$Y_n$ $\text{mod }m$  given $\cE$, and set 
\begin{equation}\label{Z n def}
 Z_n=Y_n\text{ mod m}-j(Y_n,m|\cE),\quad
\bar Z_n=Z_n-\bbE[Z_n|\cE]
\end{equation}
and for all $n<k$,
\[
S_{n,k} {\bar Z}=\sum_{s=n}^{k-1}\bar Z_s.
\]
 Next, 
\begin{lemma}
 There is a constant $b>0 $ such that for every $p\in[1,2]$ we have
\begin{equation}\label{YBp}
E\left[|S_{n,k}\bar Z|^p|\cE\right]\leq b\sum_{s=n}^{k-1}q_s(m|\cE).
\end{equation}    
 \end{lemma}

\begin{proof}
Since $|x|^p\leq |x|+|x|^2$ for all real $x$ and $p\in(1,2)$ it is enough to prove \eqref{YBp} 
 for $p=1$ and $p=2$. 
 For $p=1$, first, it follows from the triangle inequality that 
$$
\bbE[|\bar Z_s|\big|\cE]\leq 2\bbE[|Z_s|\big|\cE].
$$
Second, we have
$$
\bbE[|Z_s||\cE]=\sum_{k=0}^{m-1}|k-j(Y_s,m|\cE)|\bbP(Y_s=k\text{ mod }m|\cE)\leq m\cdot (2m)\max_{0\leq k<m}\bbP(Y_s=k\text{ mod }m|\cE)
$$
$$
\leq 2K^2\bbP(Y_s=j(Y_s,m|\cE)|\cE)=2K^2q_s(m|\cE).
$$

In order to prove the estimate with $p=2$,  we note that
by Lemma \ref{CondChain} the conditioned chain is uniformly elliptic with deterministic ellipticity constants. 
Hence, by Lemma~\ref{MomEst}(i) there is a constant $C'>0$ such that 
$$
\text{Var}(S_{n,k}Z|\cE)\leq C'\sum_{s=n}^{k-1}\text{Var}(Z_s|\cE).
$$
Next, we have 
$$
\text{Var}(Z_s|\cE)\leq \bbE(Z_s^2|\cE)=\sum_{j=0}^{m-1}(k-j(Y_s,m|\cE))^2\bbP(Y_s=k\text{ mod }m|\cE)
$$
$$
\leq m\cdot (2m)^2\cdot \max_{0\leq k<m}\bbP(Y_s=k\text{ mod }m|\cE)
\leq 4K^3\bbP(Y_n=j(Y_s,m|\cE)|\cE)=4K^3q_s(m|\cE).
$$
Therefore 
$\DS
\text{Var}(S_{n,k}Z|\cE)\leq 4K^3\sum_{s=n}^{k-1}q_s(m|\cE)
$
and the proof of \eqref{YBp} is complete. 
\end{proof}

\begin{proposition}
There is a constant $r_0>0$ so that for almost every realization  $x=(x_{n_s})_s$ of $X(\cB_N)$, for every complex number $z$ with $|z|\leq r_0$ there are nonzero complex numbers $\la_{t,n,x}(z)$, uniformly bounded functions $h_{t,n,x}^{(z)}$ and  uniformly bounded
complex measures $\nu_{t,n,x}^{(z)}$  so that $\nu_{t,n,x}^{(z)}(\textbf{1})=1$, 
$|\la_{t,n,x}(z)-1|<\frac12$, $\|h_{t,n,x}^{(z)}-1\|_{\infty}\leq \frac12$ and for all bounded functions $g$, $n\geq1$ and $p>0$ we have
\begin{equation}\label{ExpConv}
\left\|\bbE[e^{itS_{n,p}Z+zS_{n,p}}g(X_{n+p})|X(\cB_N)=x]/\la_{t,n,p,x}(z)-\nu_{t,n+p,x}^{(z)}(g)\int h_{t,n,x}^{(z)}d\mu_{t,n,x}\right\|_\infty
\end{equation}
$$
\leq A_0\|g\|_\infty\gamma^p
$$
where $\gamma\in(0,1)$ and $A_0>0$ are constants which do not depend on $x$ and $N$, and $\mu_{n,x}$ is the law of $X_n$ given that $X(\cB_N)=x$. 
Moreover,
$\DS
\la_{t,n,p,x}(z)=\prod_{k=n}^{n+p-1}\la_{t,k,x}(z).
$    
\end{proposition}

\begin{proof}
Recall that by Lemma \ref{CondChain}, 
for each $s$ the Markov chain 
$\{X_n: n_s<n<n_{s+1}\}$  conditioned on $\cE=\sig\{X(\cB_N)\}$ 
is uniformly elliptic with constants not depending on $s$ and $\cE$ 
(where $\cB_N=\{n_1<n_2<...<n_{k_N}\}$).
Moreover, the processes $\textbf{X}_s=\{X_n: n_s<n<n_{s+1}\}$ are independent, where we set $n_0=0$ and $n_{N_0+1}=\infty$.   Consider the family of functions 
$g_n^{(z)}=z Y_n+it Z_n$, 
where $z$ is a complex valued parameter and $t$ is the given non-zero $K$--resonant point
(recall that both $Y_n$ and $Z_n$ are functions of $X_n$).
Then, since the $L^2$-norm of $Z_n$ is $O(\sqrt{\bar\eps})$ 
(both unconditionally and conditionally on $\cE$), 
 we can apply Theorem \ref{RPF} after conditioning on $\cE$ 
provided that $\bar\eps$ is small enough.
\end{proof}
Next
set $\Pi_{t,n,x}(z)=\ln \lambda_{t,n,x}(z)$ and 
$\DS \Pi_{t,n,p,x}=\sum_{s=n}^{n+p-1}\Pi_{t,s,x}(z).$
In what follows we will investigate the properties of the functions $\Pi_{t,n,p,x}(z)
$ when considered as functions of $z$ (where $t$ and $x$ are fixed). First, taking the logarithms  in \eqref{ExpConv},
we see that there is a constant $C>0$ such that
\begin{equation}\label{LogPi}
\left|\Pi_{t,n,p,x}(z)-\Gamma_{t,n,p,x}(z)\right|\leq C
\end{equation}
where 
\begin{equation}\label{Gamm def}
\Gamma_{t,n,p,x}(z)=\ln\bbE[e^{itS_{n,p}Z+zS_{n,p}}|X(\cB_N)=x].
\end{equation}
Let us also set $\Gamma_{t,N,x}(z)=\Gamma_{t,1,N,x}(z)$.
Using the analyticity in $z$ (and the uniform bounds in $t$ and $x$), together with the Cauchy integral formula, we get that for every $u\in\bbN$ there is a constant $C_u>0$, such  that for every complex number $z$ with $|z|\leq r_0/2$  
 we have
\begin{equation}\label{Pi est}
\left|\Pi_{t,n,p,x}^{(u)}(z)-\Gamma_{t,n,p,x}^{(u)}(z)\right|\leq C_u
\end{equation}
uniformly in $t,n,p$ and $x$. Here $g^{(u)}(z)$ denotes the $u$-th derivative of a function $g(z)$.

\subsection{The derivatives of the conditional cumulant generating function}

 We  need the following result.
\begin{lemma}\label{Pi'' lemma}
Uniformly in $t$ and $x$
for every
 $p\in(1,2)$ we have the following:

(i) 
$\DS
\Pi_{t,1,N,x}'(0)=\bbE(S_N)+O\left((\ln\sig_N)^{\frac{2}{p+2}}(\sig_N)^{\frac{2p}{p+2}})\right)=O\left((\ln\sig_N)^{\frac{2}{p+2}}(\sig_N)^{\frac{2p}{p+2}})\right)
$
and so
$$
\Gamma_{t,N,x}'(0)=\bbE(S_N)+O\left((\ln\sig_N)^{\frac{2}{p+2}}(\sig_N)^{\frac{2p}{p+2}})\right)=O\left((\ln\sig_N)^{\frac{2}{p+2}}(\sig_N)^{\frac{2p}{p+2}})\right).
$$

(ii) We have
$\DS
\Pi_{t,1,N,x}''(0)=\sigma_N^2+O\left((\sig_N)^{\frac{2p}{p+1}}(\ln \sig_N)^{\frac{1}{p+1}}\right)
$
and so 
\begin{equation}\label{Pi1}
\Gamma_{t,N,x}''(0)=\nonumber \sigma_N^2+O\left((\sig_N)^{\frac{2p}{p+1}}(\ln \sig_N)^{\frac{1}{p+1}}\right).
\end{equation}

(iii) For every integer $u\geq 3$ there exist  constants $D_u>0$ and $\del_u>0$ 
(independent of $t$ and $x$) such that for every $N$ and all $h\in[-\del_u, \del_u]$ we have
$$
\left|\Pi_{t,1,N,x}^{(u)}(ih)\right|\leq D_u\sigma_N^2.
$$
Consequently, there is a constant $D_u'$ such that
$$
\left|\Gamma_{t,N,x}^{(u)}(ih)\right|\leq D'_u\sigma_N^2.
$$
\end{lemma}
\begin{proof}
Denote $\Pi_{n}(z)=\Pi_{t, n,x}(z)$.
For every finite set $I\subset\bbN$ we set $\DS \Pi_I(z)=\sum_{n\in I}\Pi_{n}(z)$.

 First we prove  (ii).
Let $\ve_N<\bar\epsilon$. 
Let us distinguish between two types of blocks $I_s=[n_s,n_{s+1})$. The first are the blocks $I_s$ such that 
$\DS \sum_{n_s\le n\le n_{s+1}}q_n(m|\cE)\leq 4\ve_N$ and the second type are the blocks $I_s$ such that  
$\DS \sum_{n_s\le n\le n_{s+1}}q_n(m|\cE)>4\ve_N$. Let us further divide each one of the blocks $I_s$ of the second type into blocks $I_{k,s}$ such that for each $k$ we have 
$\DS \ve_N\leq \sum_{n\in I_{k,s}}q_n(m|\cE)\leq 2\ve_N$. Let us consider the block partition of $\{1,2,...,N\}$ generated by the blocks of type one $I_s$ and the sub-blocks $I_{k,s}$ of all the blocks of the second type. Denote this partition by $B_1,...,B_{l_N}$. Then each one of the blocks $B_j$ is a subset of one of the blocks $I_s$. Moreover, the number of blocks $l_N$ satisfies $l_N=O(\ln \sig_N/\ve_N)$ uniformly in the realizations $x$ of $\cE$.
 Indeed, by assumption we have $\tilde M_N(m)=O(\ln \sig_N)$, and using Lemma \ref{FiberQuantLemma} we get the same estimates after conditioning on $\cE$.
Thus, the number of blocks $I_{k,s}$ with $I_s$ of the second type is $O(M_N(m)/\varepsilon_N)=O(\ln\sig_N/\ve_N)$. Taking into account that the total number of blocks $I_s$ is $O(M_N(m)/\bar\epsilon)$ we get that 
\begin{equation}
\label{EllN-Log}
l_N=O(\ln\sig_N/\ve_N).
\end{equation}

Next, by \eqref{LogPi} for every $k$ we have
\[
\Pi_{B_k}''(0)=\frac{\bbE[e^{it_j S_{B_k} Z}S_{B_k}^2|\cE]}{\bbE[e^{it_j S_{B_k}Z}|\cE]}-
\left(\frac{\bbE[e^{it_j S_{B_k}Z}S_{B_k}|\cE]}{\bbE[e^{it_j S_{B_k}Z}|\cE]}\right)^2+O(1).
\] 
$$
=\frac{\bbE[e^{it_j S_{B_k}\bar Z}S_{B_k}^2|\cE]}{\bbE[e^{it_j S_{B_k}\bar Z}|\cE]}-
\left(\frac{\bbE[e^{it_j S_{B_k}\bar Z}S_{B_k}|\cE]}{\bbE[e^{it_j S_{B_k}\bar Z}|\cE]}\right)^2+O(1).
$$
Here $\bar Z_n=Z_n-\bbE[Z_n|\cE]$ (like in \eqref{Z n def}) and $S_{B}=\sum_{j\in B}f_j(X_j)$ for every finite set $B\subset\bbN$.
To estimate the first term on the right hand side,  we decompose
$$
\frac{\bbE[e^{it_j S_{B_k}\bar Z}S_{B_k}^2|\cE]}{\bbE[e^{it_j S_{B_k}\bar Z}|\cE]}=
\frac{\bbE[(e^{it_j S_{B_k}\bar Z}-1)(S_{B_k}^2-\bbE[S_{B_k}^2|\cE])|\cE]}
{\bbE[e^{it_j S_{B_k}\bar Z}|\cE]}+\bbE[S_{B_k}^2|\cE].
$$
By \eqref{YBp} the first summand on the right hand side satisfies
$$
\bbE\left[|S_{B_k} \bar Z|\,\big|\cE\right]\leq C\sum_{n\in B_{k}}q_n(m|\cE)\leq 4C\ve_N.
$$
Therefore, when $\ve_N$ is smaller than some sufficiently small constant $c_0>0$ then
$$
\left|\bbE[e^{it_j S_{B_k}\bar Z}|\cE]-1\right|\leq \frac12
$$ 
which implies that 
$$
|\bbE[e^{it_j S_{B_k}\bar Z}|\cE]|\geq \frac12.
$$
Next, let us estimate the numerator. Since
\begin{equation}\label{Y B}
|e^{it_j S_{B_k}\bar Z}-1|\leq |t_j||S_{B_k}\bar Z|
\end{equation}
for every $p\in(1,2)$ we have
$$
\left|\bbE[(e^{it_j S_{B_k}\bar Z}-1)(S_{B_k}^2-\bbE[S_{B_k}^2|\cE])|\cE]\right|\leq C\|S_{B_k}\bar Z\|_{L^p(\cE)}
\left(\|S_{B_k}^2\|_{L^{q}(\cE)}+|\bbE[S_{B_k}^2|\cE]|\right)
$$
where $L^q(\cE)$ denotes the $L^q$ norm with respect to the conditional measure and $q$ is the conjugate exponent of $p$.
To estimate $\|S_{B_k}^2\|_{L^q(\cE)}$, 
let  $q_0=[q]+1$. Then
$$
\bbE[S_{B_k}^{2q_0}|\cE]=
\bbE[\left(\bar S_{B_k}+\bbE[S_{B_k}|\cE]\right)^{2q_0}|\cE]
=\sum_{j=0}^{2q_0}\binom{2q_0}{j}\bbE[\bar S_{B_k}^j|\cE]
(\bbE[S_{B_k}|\cE])^{2q_0-j}.
$$
By \eqref{E bound} we have
$\DS \sup_{a<b}|\bbE[S_b]-\bbE[S_a]|<\infty$, 
and therefore  $\bbE[S_{B_k}|\cE]$ is uniformly bounded in $k$. 
Applying the moment estimates  of Lemma \ref{MomEst}
to the conditioned Markov chain we see that for every integer $w\geq1$ there is  a constant $C_w\geq 1$ (independent of the realization $x$ of $\cE$) so that 
\begin{equation}
\label{Norms}
\|\bar S_{B_k}\|_{L^w(\cE)}\leq C_w(1+\|\bar S_{B_k}\|_{L^2(\cE)}).
\end{equation}
We thus conclude that 
$$
\|S_{B_k}^2\|_{L^q(\cE)}\leq \|S_{B_k}^2\|_{L^{q_0}(\cE)}=
 \|S_{B_k}\|_{L^{2q_0}(\cE)}^2
=O(1+\|\bar S_{B_k}\|_{L^2(\cE)}^2).
$$
To estimate the term $\bbE[S_{B_k}^2|\cE]$ we have 
$$
\bbE[S_{B_k}^2|\cE]=V(S_{B_k}|\cE)+(\bbE[S_{B_k}|\cE])^2.
$$
The $L^\infty$ norm of the second term is $O(1)$ because of \eqref{Expectation}. Combining the above estimates and using \eqref{YBp} we see that there is a constant $C>0$ such that, a.s. we have
$$
\left|\frac{\bbE[(e^{it S_{B_k}\bar Z}-1)(S_{B_k}^2-\bbE[S_{B_k}^2|\cE])|\cE]}
{\bbE[e^{it S_{B_k}\bar Z}|\cE]}\right|\leq C{(\ve_N)^{1/p}}(1+V(S_{B_k}|\cE)).
$$
We conclude that for every $p\in(1,2)$ we have
$$
\frac{\bbE[e^{itS_{B_k}Z}S_{B_k}^2|\cE]}{\bbE[e^{it S_{B_k}Z}|\cE]}=\bbE[S_{B_k}^2|\cE]+
O\big((\ve_N)^{1/p}\big)\text{Var}(S(B_k)|\cE)+
O\big((\ve_N)^{1/p}\big)
$$
uniformly in the realizations of $\cE$.
Similar arguments show that for every $p\in(1,2)$ 
\begin{equation}\label{Sim}
\left|\frac{\bbE[e^{it S_{B_k} Z} S_{B_k}|\cE]}{\bbE[e^{it S_{B_k}Z}|\cE]}-\bbE[S_{B_k}|\cE]\right|=\left|\frac{\bbE[(e^{it S_{B_k}\bar Z}-1)(S_{B_k}-\bbE[S_{B_k}|\cE])|\cE]}{\bbE[e^{it S_{B_k}\bar Z}|\cE]}\right|
\end{equation}
$$
=O\big((\ve_N)^{1/p}\big)\sqrt{\text{Var}(S(B_k)|\cE)}.
$$ 

Now, by \eqref{Expectation} and since each $B_k$ is contained in one of the blocks 
$[n_s,n_{s+1})$ we have 
$$
\left|\bbE[S_{B_k}|\cE]-\bbE(S_{B_k})\right|\leq C, \text{a.s.}
$$
for some constant $C>0$.
Since $\sup_k|\bbE(S_{B_k})|<\infty$ (by \eqref{E bound}) we get that 
there is a constant $C_1>0$ such that for every $k$ we have
$$
\|\bbE[S_{B_k}|\cE]\|_{L^\infty}\leq C_1.
$$
Combining the above estimates we derive that for every $p\in(1,2)$ we have
\begin{equation}\label{pi k est}
\Pi_{B_k}''(0)=\text{Var}(S_{B_k}|\cE)(1+O\left(\ve_N)^{1/p})\right)+
O\left((\ve_N)^{1/p}\right)\sqrt{\text{Var}(S_{B_k}|\cE)}+O(1)
\end{equation}
uniformly.
Next, set $ \mathfrak{s}_k=\text{Var}(S(B_k)|\cE)$. Then 
\begin{equation}\label{sk}
\sum_{k=1}^{l_N}\sqrt{\mathfrak{s}_k}\leq\sqrt{l_N}
\sqrt{\sum_{k=1}^{l_N} \mathfrak{s}_k}\leq A\left(\frac{\ln\sig_N}{\ve_N}\right)^{1/2}\sig_N
\end{equation}
for some constant $A$.
Therefore, 
\begin{equation}\label{ff}
(\ve_N)^{1/p}\sum_{k=1}^{l_N}\sqrt{\text{Var}(S_{B_k}|\cE)}\leq C \sig_N(\ve_N)^{1/p-1/2} \sqrt{\ln\sig_N} 
\end{equation}
for some constant $C>0$.
Next, by  Lemma \ref{CondChain}, the Markov chain $(X_n)_{n\geq 1}$ 
is still uniformly elliptic after conditioning on $\cE$, with deterministic ellipticity constants. Thus,
also the exponential decay of correlation  
(Lemma \ref{DEC}) holds  with deterministic constants  after such a conditioning. Consequently,
\begin{equation}\label{v s N}
\text{Var}(S_N|\cE)=\sum_{k=1}^{l_N}\text{Var}(S_{B_k}|\cE)+2\sum_{1\leq k<\ell\leq l_N}\text{Cov}(S_{B_k},S_{B_\ell}|\cE)=
\sum_{k=1}^{l_N}\text{Var}(S_{B_k}|\cE)+O(l_N)
\end{equation}
where in the last equality we have used that 
$$
\left|2\sum_{1\leq k<\ell\leq l_N}\text{Cov}(S_N-S_{B_k},S_{B_k}|\cE)\right|\leq\sum_{k=1}^{l_N}\left|\text{Cov}(S_N-S_{B_k},S_{B_k}|\cE)\right|
$$
and that because of the exponential decay of correlations of the conditioned chain there is a constant $C'>0$ such that  for all $k$ and $N$,
$$
\text{esssup}\left|\text{Cov}(S_N-S_{B_k},S_{B_k}|\cE)\right|\leq C'.
$$
Since $l_N=O(\ln\sig_N/\ve_N)$, starting from \eqref{pi k est} and using \eqref{ff} and \eqref{v s N}, we conclude that 
\begin{equation}\label{Pi sum Conc}
\sum_{n=1}^N\Pi_n''(0)=\sum_{k=1}^{l_N}\Pi_{B_k}''(0)=
\end{equation}
$$O(\ln\sig_N/\ve_N)+(1+O\big((\ve_N)^{1/p})\big)(\text{Var}(S_N|\cE)+O(\ln\sig_N/\ve_N))+
O\left(\sig_N(\ve_N)^{1/p-1/2} \sqrt{\ln\sig_N} \; \right)
$$
$$=\text{Var}(S_N|\cE)+O\big((\ve_N)^{1/p}\big)\text{Var}(S_N|\cE)+O(\ln\sig_N/\ve_N)+
O\left(\sig_N(\ve_N)^{1/p-1/2} \sqrt{\ln\sig_N}\; \right).$$

This together with (\ref{VarDiff}) and the choice 
$\ve_N=\sig_N^{-\frac{2p}{p+1}} \big(\ln\sig_N\big)^{\frac{p}{p+1}}$ yields (ii), where we have  used that,  by \eqref{VarRel},
$\text{Var}(S_N|\cE)/\sigma_N^2$ is uniformly bounded and bounded away from $0.$ 
\smallskip

Now we derive (i) using the estimates  obtained in the proof of (ii). 
By \eqref{Sim} and \eqref{ExpDiff}, 
$$
\Pi_{1,N}'(0)=\bbE(S_N)+O(\ln\sig_N/\ve_N)+
O\left((\ve_N)^{1/p}\right)\sum_{k=1}^{l_N}\sqrt{\text{Var}(S(B_k)|\cE)}.
$$
Next, by \eqref{sk}, 
$$
\Pi_{1,N}'(0)=\bbE(S_N)+O(\ln\sig_N/\ve_N)+
 O\left((\ve_N)^{1/p} \left(\frac{\ln\sig_N}{\ve_N}\right)^{1/2}\sig_N\right).
$$
Taking $\ve_N=\left(\frac{\ln\sig_N}{\sig_N^2}\right)^{\frac{p}{p+2}}$ we get (i).
\smallskip

In order to prove (iii), we first claim that for every $c_0>0$ small enough the following condition holds.  For all $n<k$ with $\DS \sum_{s=n}^{k-1}q_s(m|\cE)\leq c_0$, we have 
\begin{equation}\label{LB}
|\bbE(e^{it S_{n,k}}|\cE)|=|\bbE(e^{it S_{n,k} \bar Z}|\cE)|\geq \frac12.
\end{equation} 
To prove the claim we apply \eqref{YBp} with $p=2$ and use that for every zero mean random variable $X$ we have 
$|\bbE[e^{itX}]-1|\leq |t|^2\|X\|_{L^2}$ (we take $X$ to be $S_{n,k}|\cE$). Thus, we can take $c_0\leq \frac{1}{2t^2b}$, where $b$ comes from \eqref{YBp}.

Let $c_0>0$ be sufficiently small. Fix some $s$ such that $\DS \sum_{n_s<n<n_{s+1}}q_n(m|\cE)\geq c_0$, and decompose $I_s=\{n_s\leq n<n_{s+1}\}$ into  blocks $B_{s,1},B_{s,2},...,B_{s, L_s}$ so that for each $k$ we have 
$\DS \frac{c_0}{4}<\sum_{n\in B_{s,k}}q_n(m|\cE)<\frac{c_0}{2}$ (this is possible if $\bar\epsilon$ is small enough). 
{Note that the number of blocks $L_s$ is $O(\ln \sig_N)$.  If   $\DS \sum_{n_s<n<n_{s+1}}q_n(m|\cE)< c_0$ then we define $L_s=1$ and $I_s=B_{s,1}$. By considering all the blocks $B_{s,j}, j\leq L_s$ we get a partition of $\{1,2,...,N\}$ into blocks $B_{1},.... B_{L_N}$ such that for every $k$ we have $\DS \sum_{n\in B_k}q_n(m|\cE)\leq c_0$. 

Next, fix a large constant $A$. If $\text{Var}(S_{B_j}|\cE)$ is larger than  $2A$ then we subdivide the block $S_{B_j}$ into smaller blocks  so that the variance along each blocks is between $A$ and $2A$ (the arguments are the same as the ones appearing after \eqref{Therefore}). 
We conclude that there is a partition of $\{1,2,...,N\}$ into blocks $\tilde B_1,...,\tilde B_{\tilde L}$ so that $\tilde L\asymp \sig_N^2$, the 
 conditional variances along the blocks are uniformly bounded  by $2A$ and 
\begin{equation}\label{LB1}
|\bbE(e^{it S_{\tilde B_s}}|\cE)|\geq \frac12.
\end{equation}
for each block $\tilde B_s$. 
To see why  $\tilde L\asymp \sig_N^2$, we  first note that $L_s\!\!\leq\!\!R'\ln \sig_N$.  Therefore,  the sum of the variances of the  blocks $B_{s,j}$ for which 
$\text{Var}(S_{B_{s,j}}|\cE)<A$  is of order $O(\ln \sig_N)$, and so the sum of the variances of the  blocks $B_{j}$ for which 
$\text{Var}(S_{B_j}|\cE)<A$ is of order $O(\ln^2 \sig_N)$. Hence, using also Lemma \ref{MomEst} (after conditioning on $\cE$), we see that  the total variance $\text{Var}(S_N|\cE)$ (which is of order $\sig_N^2$) has the same order as the sum of the variances $\text{Var}(S_{B_j}|\cE)$ such that $\text{Var}(S_{B_j}|\cE)\in[A,2A]$. However, this sum is order $\tilde L$.

Next, since $\text{Var}(S_{\tilde B_s}|\cE)\leq 2A$, by the mean value theorem
we have 
$$
\left|\bbE(e^{it S_{\tilde B_s}+ih S_{\tilde B_s}}|\cE)-\bbE(e^{it S_{\tilde B_l}}|\cE)\right|\leq 4A|t||h|.
$$
Combinning this with \eqref{LB1} we see that
 there exists $0<h_0<r_0/2$ so that for every $h\in[-h_0,h_0]$ we have  
\begin{equation}\label{LB2}
\text{essinf}\,|\bbE(e^{it S_{\tilde B_l}+ihS_{\tilde B_l}}|\cE)|\geq \frac14
\end{equation}

 Next, let us decompose $\Pi_{1,N}$ according to the blocks $\tilde B_s$:
$$
\Pi_{1,N}=\sum_{s=1}^{\tilde L}\Pi_{\tilde B_s}.
$$ 
Let us take a realization of $\cE$, namely we fix $X(\cB_N)=x$.
Differentiating both $\Pi_{\tilde B_s}$  and $\Gamma_{t,\tilde B_s,x}=\Gamma_{\tilde B_s}$\, $u$-times, and  using the Cauchy integral formula  together with \eqref{LogPi}, we see that if $|h|<h_0$ then 
\begin{equation}\label{One}
\left|\Pi_{1,\tilde B_s,x}^{(u)}(ih)\right|\leq C_u+\left|\Gamma_{\tilde B_s}^{(u)}(ih)\right|
\end{equation}
where $C_u$ is a constant which depends on $u$ but not on $x,N$ or $h$. Next, let us bound  
$$\psi(h)=\psi_s(h):=\Gamma_{t,\tilde B_j,x}^{(u)}(ih)=\ln\bbE[e^{it S_{\tilde B_s}}e^{ih S_{\tilde B_s}}|X(\cB_N)=x].$$
To ease the notation, let us abbreviate $S_{\tilde B_js}=S$ and 
$ W=t_j S_{\tilde B_s}$. Then
by Fa\'a di Bruno's formula \cite[\S 1.3]{KP}, for every $h\in[-r_0,r_0]$ we have
$$
|\psi^{(u)}(h)|
$$
$$
=\left|\sum_{(m_1,...,m_{u})}\frac{u!}{\prod_{q=1}^{u}(m_q!(q!)^{m_q})}\cdot\frac1{\psi(h)^{\sum_{l=1}^u m_q}}\prod_{w=1}^{u}\left((i)^w
\bbE[S^we^{i W+ih S}|X(\cB_N)=x]\right)^{m_w}
\right|
$$
where 
$(m_1,...,m_u)$ range over all the $u$-tuples of nonnegative integers such that $\DS \sum_{q}q m_q=u$.
Notice now that by  \eqref{LB2} the denominators appearing above are bounded below by some positive constant. Moreover, by Lemma \ref{MomEst}
 there are constants $c_p$ such that, for every $p\geq 1$, we have $\|S\|_{L^p(\cE)}\leq c_p$ a.s. Therefore, by the H\"older inequality,  the numerators are uniformly bounded above. We conclude that if $|h|<h_0$ then 
 $$
 |\psi^{(u)}(h)|\leq C(u,A)
 $$
 for some constant $C(u,A)$ which depends only on $u$ and $A$. Thus by \eqref{One},
 $$
 \left|\Pi_{1,\tilde B_s,x}^{(u)}(ih)\right|\leq C'(u,A)
 $$
 and since the number of blocks is $O(\sig_N^2)$ we conclude that 
  if $|h|<h_0$  then 
  $$
  \left|\Pi_{1,N,x}^{(u)}(h)\right|\leq C''\sig_N^2
  $$
  for some other constant $C''$ (which does not depend on $x$).}
 \end{proof}

\subsection{The canonical form of the 
generalized Edgeworth polynomials at resonant points}
Let
$t$ be a nonzero $K$--resonant point. Set
\begin{equation}
\label{DefLambda}
\Lambda_{t,N,x}(h)=\Gamma_{t,N,x}(ih/\sig_N)=\ln\bbE[e^{itS_N Y}e^{ihS_N/\sig_N}|X(\cB_N)=x]
\end{equation}
and
\begin{equation}\label{H def}
H_{t,N,x,r}(t)=
\end{equation}
$$
1+\sum_{\bar k}\frac1{k_1!\cdots k_{r}!}\left(\frac{\Lambda_{t,N,x}^{(3)}(0)}{3!}\right)^{k_1}\cdots \left(\frac{\Lambda^{(r+2)}_{t,N,x}(0)}{(r+2)!}\right)^{k_{r}}(it)^{3k_1+...+(r+2)k_{r}}
$$
where the summation runs over the collection of $r$ tuples of nonnegative integers 
$(k_1,...,k_{r})$ that are not all $0$ such that $\DS \sum_{j} jk_j\leq r$.
Then we can also write
\begin{equation}\label{H}
H_{t, N, x,r}(z)=1+\sum_{q=1}^{r}\sig_N^{-j}\tilde P_{t,N,x,q}(z)
\end{equation}
with 
\begin{equation}\label{H P}
\tilde P_{t,N,x,q}(z)=\sum_{\bar k\in A_{q}}C_{\bar k}\prod_{j=1}^{s}\left(\sig_N^{-2}\Gamma^{(j+2)}_{t,N,x}(0)\right)^{k_j}(iz)^{3k_1+...+(s+2)k_{s}}
\end{equation}
where $A_{q}$ is the set of all  tuples of nonnegative integers $\bar k=(k_1,...,k_{s})$, 
 for some $s=s(\bar k)\geq 1$
such that $\DS \sum_{l}lk_l=q$ (note that when $j\leq r$ then $s\leq r$ since $k_s\geq1$). 
 Moreover
 $\DS 
 C_{\bar k}=\prod_{j=1}^{s}\frac{1}{k_j!(j+2)^{k_j}}.
 $

By Lemma \ref{Pi'' lemma} 
the coefficients of each $\tilde P_{t_j,N,x,q}$ are uniformly bounded 
(when $x$ and $N$ vary).
Next, set
$$
\tilde\Lambda_{t,N,x}(h)=\Lambda_{t,N,x}(h)-\left(\Lambda_{t,N,x}(0)+h\Lambda_{t,N,x}'(0)+(h^2/2)\Lambda_{t,N,x}''(0)\right),$$
\begin{equation}\label{d u def}
d_N=d_{t,N,x}=\Lambda_{t,N,x}'(0)-\frac{i\bbE[S_N]}{\sig_N}
\quad\text{and}\quad
u_N=u_{t,N,x}=\Lambda_{t,N,x}''(0)-1.
\end{equation}
Then by Lemma \ref{Pi'' lemma}, for every $p\in(1,2)$ we have
\begin{equation}\label{u d}
\|d_{t,N,\cdot}\|_{L^\infty}=O(\sig_N^{-\frac{2-p}{p+2}}\ln^{\frac{2}{p+2}}\sig_N), \quad
\|u_{t,N,\cdot}\|_{L^\infty}=O\big(\sig_N^{-\frac{2}{p+1}}(\ln\sig_N)^{\frac{1}{p+1}}\big).
\end{equation}

\begin{proposition}\label{d u prop}
For every $r$ there  are constants $\del_r,C_r>0$ so that for every realization $x$ of  $X(\cB_N)$ and every real $h$ with $|h|\leq \del_r \sig_N$ we have
$$
\bbE\left(e^{i(t+h/\sig_N)S_N}|X(\cB_N)=x\right)
$$
$$
=\bbE(e^{it S_N}|X(\cB_N)=x)e^{-h\Lambda_{t,N,x}'(0)+(h^2/2)\Lambda_{t,N,x}''(0)}H_{t,N,x,r}(h)
+\theta_{N,r,x}
\sig_N^{-r-1}e^{-h^2/4}
$$
where $\theta_{N,r,x}$ satisfy $\sup_{N,x}|\theta_{N,r,x}|<\infty$. As a consequence,
\begin{equation}
\label{PreExpRes}
\bbE\left(e^{i(t+h/\sig_N)S_N}|X(\cB_N)=x\right)
=\bbE(e^{it S_N}|X(\cB_N)=x)e^{-ih\bbE[S_N]/\sig_N}e^{-h^2/2}\times
\end{equation}
$$
\left(e^{hd_N+h^2u_N/2}
e^{-h\Lambda_{t,N,x}'(0)+h^2/2\Lambda_{t,N,x}''(0)}H_{t,x,r}(h)\right)
+\theta_{N,r,x}
\sig_N^{-r-1}e^{-h^2/4}.
$$

\end{proposition}
\begin{proof}
We have 
$$
\bbE\left(e^{i(t+h/\sig_N)S_N}|X(\cB_N)=x\right)=\exp\left(\Lambda_{t,N,x}(h)+
it\sum_{n=1}^{N}j(Y_n,m|\cE)\right)=
$$
$$
\bbE(e^{it S_N}|X(\cB_N)=x)e^{-h\Lambda_{t,N,x}'(0)+(h^2/2)\Lambda_{t,N,x}''(0)}\exp(\tilde\Lambda_{t,N,x}(h)).
$$
Next, by Lemma \ref{Pi'' lemma} (iii) we have
$\DS \tilde\Lambda_{t,N,x}^{(q)}(0)=0,\,q=0,1,2$
and 
for $j\geq3$,
$$
\sup_{x}\left|\tilde\Lambda_{t,N,x}^{(j)}(h)\right|=\sup_x\left|\Lambda_{t,N,x}^{(j)}(h)\right|=O(\sig_N^{-(j-2)}).
$$
Notice also that by by Lemma \ref{Pi'' lemma} (ii) we have $\sup_x|\Lambda_{t,N,x}''(0)-1|\to 1$ as $N\to\infty$. 
 Now the proof of the proposition is completed using Proposition \ref{LamProp}, applied for every realization
 $X(\cB_N)=x$.
\end{proof}

\subsection{Proof of Theorem \ref{MainThm}}
Recall the decomposition \eqref{SplitInt} of $\bbP(S_N=k)$. In this section we will expand the integrals 
$\DS \sum_{j}\int_{I_j} e^{-ixk}\bbE(e^{ix S_N})dx$ for $K$--resonant points $t_j=\frac{2\pi l}{m}$ so that $M_N(m)\leq R\sig_N$ for some constant $R$.

Let us fix such a point $t_j$ and denote $t_j=t$.
Recall first that
$$
\int_{-\infty}^{\infty} e^{-i\al h}e^{-h^2/2}h^k dh=2\pi(-1)^kH_k(\al) \mathfrak{g}(\alpha)
$$
where $H_k$ is the $k$-th Hermite polynomial  and $\mathfrak{g}$ is the standard Gaussian 
density. 
Now, with $I=[t-\del,t+\del]$ for some $\del>0$ write 
$$
\int_{I}e^{-ik s}\bbE(e^{is S_N}) ds=\bbE\left[e^{-itk}\sig_N^{-1}
\int_{-\del\sigma_N}^{\del \sigma_N}e^{-ikh}\bbE\left(e^{i(t_j+h/\sig_N)S_N}|\cE\right)dh\right].
$$

Using the Lagrange form for the Taylor expansions around  the origin of the function $e^x$  and \eqref{u d}, we can
expand the terms $e^{hd_N}$ and $e^{-h^2 u_N/2}$ 
 in \eqref{PreExpRes} and  
get that   
$$
\int_{I}e^{-ik s}\bbE(e^{is S_N}) ds=\bbE[\cI_t]+o\left(\sig_N^{-r}\right)
\quad\text{where}
$$$$
\cI_t
=\sig_N^{-1}\bbE(e^{it_j S_N}|\cE)
$$
$$
\times \int_{-\infty}^\infty { e^{-i (t k+hk/\sigma_N)}}
\left(1+\sum_{j=1}^{3r-2}\frac{h^jd_N^j}{j!}\right)\!\!
\left(1+\sum_{j=1}^{r}\frac{h^{2j}u_N^j}{j!2^j}\right)H_{t_j,N,X(\cB_N),r}(h)e^{-h^2/2}dh.
$$
Next, let us fix $X(\cB_N)=x$. Then, 
it follows from the definition of $H_{t,N,x,r}$ that
\begin{equation}\label{A def}
\left(1+\sum_{j=1}^{3r-2}\frac{h^jd_N^j}{j!}\right)
\left(1+\sum_{j=1}^{r}\frac{h^{2j}u_N^j}{j!2^j}\right)H_{t,N,x,r}(h)=1+\sum_{s=1}^{w_r}A_{t,s,N,x}h^s+g_{N,r,x}(h)
\end{equation}
where  $w_r=5r-2$ and
$g_{N,r,x}(h)$ is a polynomial whose coefficients are $o(\sig_N^{-r-1})$, uniformly in $x$. Thus, denoting $ \hat k_N=\frac{k-\bbE[S_N]}{\sig_N}$, using the above estimates we see that 
\begin{equation}
\label{ResOscEdge}
\int_{I}e^{-ik s}\bbE(e^{is S_N}) ds=e^{-itk}
\bbE[\cJ_t]+o(\sig_N^{-r})
\end{equation}
where
\begin{equation}\label{ExForm}
\cJ_t=\sig_N^{-1}\bbE(e^{it S_N}|X(\cB_N))\varphi(\hat k_N)\left(1+\sum_{s=1}^{w_r}A_{t,s,N,X(\cB_N)}(-1)^s H_{s}(\hat k_N)\right).
\end{equation}
\begin{proof}[Proof of Theorem \ref{MainThm}]
The theorem follows from  \eqref{H}, \eqref{ResOscEdge}, \eqref{ExForm}, and 
 Proposition \ref{LLL1}  (that shows that the contribution of  nonzero $K-$resonant points is negligible when 
$M_N(m)\geq R\ln\sig_N$  with $R$ large enough).
\end{proof} 

\section{Classical Edgeworth expansions-proof of Theorem \ref{Thm Stable Cond}}\label{Sec6}

\begin{proposition}
The condition $\bbE[e^{itS_N}]=o(\sig_N^{-(r-1)})$ 
 for all nonzero $K$--resonant points $t$ is necessary for the usual expansions of order $r$ to hold.
\end{proposition}
\begin{proof}
Let us write 
\begin{equation}\label{Poly}
1+\sum_{s=1}^{w_r}A_{t,s,N,\cE}(-1)^sH_{s}(\hat k_N)=\sum_{u=0}^{w_r}B_{t,u,N,X(\cB_N)}\hat k_N^u.
\end{equation}
Using  \eqref{ExForm} and
\cite[Lemma 5.1]{DH} we conclude that $S_N$ obeys the Edgeworth expansion of order $r$  iff
\begin{equation}\label{BBB}
\bbE[\bbE(e^{it S_N}|\cE)B_{t,u,N,X(\cB_N)}]=o(\sig_N^{-(r-1)})
\end{equation}
for all $u$, 
and $t$ like in the statement of this proposition. Note that the role of \cite[Lemma 5.1]{DH} is that it  yields that the classical expansion holds iff the contribution of all nonzero $K$--resonant points is $o(\sig_N^{-r})$ and, in fact, iff 
 all the coefficients of $\hat k_n^u$ are $o(\sig_N^{-r})$. These coefficients are the numbers appearing on the left hand side of \eqref{BBB}.

 Next, let us suppose that the Expansion of order $r$ holds. 
We claim that 
\begin{equation}\label{AAA}
\bbE[\bbE(e^{it S_N}|\cE)A_{t,s,N,X(\cB_N)}]=o(\sig_N^{-(r-1)})
\end{equation}
for all $s$.
Indeed, \eqref{BBB} holds. Then, since $H_{k}$ is of degree $k$, 
the leading coefficient  on the right hand side of \eqref{Poly} is $A_{t,w_r,N,X(\cB_N)}$, which yields that 
$$
\bbE[\bbE(e^{it S_N}|\cE)A_{t,w_r,N,X(\cB_N)}]=o(\sig_N^{-(r-1)}).
$$ 
Thus, \eqref{Poly} holds with $w_r-1$ instead of $w_r$ and so \eqref{AAA} follows by induction on the number of summands.

This yields that 
$$
\int_{I}e^{-ik s}\bbE(e^{is S_N}) ds=\sig_N^{-1}e^{it k}\bbE[\bbE(e^{it S_N}|\cE)]=\sig_N^{-1}e^{itk}\bbE[e^{it S_N}]+o(\sig_N^{-r}).
$$
However, since we assumed that the classical expansion holds, by \cite[Lemma 5.1]{DH} the above integral must be $o(\sig_{N}^{-r})$
and thus 
$\DS
\bbE[e^{it S_N}]=o(\sig_N^{-(r-1)}).
$
\end{proof}
\begin{remark}\label{Rem}
The proof shows that if the conditional expansions of order $r$ hold, then 
for all $\ell$ and a $K-$resonant point $t$,
$$
\lim_{N\to\infty}\sig_N^{r-1}\sup_{j_1,...,j_\ell}\left\|\bbE[e^{it S_N}|X_{j_1},...,X_{j_\ell}]\right\|_{L^\infty}=0.
$$
Indeed we can just replace the chain by the chain conditioned on $X_{j_1},...,X_{j_\ell}$ and use that the error term in the definition of the conditional expansion depends only on $r,\sig_N$ and the number of conditioned variables $\ell$.
\end{remark}

Our next result concerns sufficient conditions for the conditional expansions.
\begin{proposition}\label{Suff}
The condition
\begin{equation}\label{CondB}
\max_{k\leq 8r-4}\, \sup_{j_1,...,j_{k}\in\cB}\|\mathbb E[e^{it S_N}|X_{j_1},...,X_{j_{k}}]\|_{L^1}=o(\sigma_N^{-(r-1)})
\end{equation}
for all $K-$resonant points $t$
is sufficient for the usual expansions of order $r$ to hold. Thus, the condition
 that for each $\ell$ and a $K-$resonant point $t$
$$
\sup_{j_1,...,j_{\ell}}\|\mathbb E[e^{it S_N}|X_{j_1},...,X_{j_{\ell}}]\|_{L^1}=o(\sig_N^{-(r-1)})
$$
is sufficient for the conditional expansions of order $r$ to hold. 
\end{proposition}

\subsection{Proof of Proposition \ref{Suff}}
Fix a realization $X(\cB_N)=x$. 
Set 
$$
 G_{k,1}=G_{t,k,1,x}=\Gamma_{t,n_k, n_{k+1},x}'(0)-i\bbE[S_{n_k, n_{k+1}}],
$$
\begin{equation}\label{Gk2}
G_{2,k}=G_{t,2,k,x}=\Gamma_{t_j,n_k,n_{k+1},x}''(0)-\text{Var}(S_{n_k,n_{k+1}}|X(\cB_N)=x)
\end{equation}
and for  $s\geq 3$, set
$$
G_{k,s}=G_{t,k,s,x}=(i)^s\Gamma_{t,n_k, n_{k+1},x}^{(s)}(0).
$$

\begin{lemma}
 The coefficients  of the generalized Edgeworth polynomials (see \eqref{ExForm}) are linear combinations of expressions of the form
\begin{equation}\label{Form}
A_N\bbE[e^{it_jS_N}]+\sum_{1\leq k_1,...,k_{\ell_r}\leq N_0}c_{k_1,...,k_{\ell_r}}\bbE\left[\bbE[e^{it_jS_N}|\bar X_{k_{1}},...,\bar X_{k_{\ell_r}}]G_{k_1,...k_\ell,N}\right]
\end{equation}
where $\ell_r=4r-2$, with $\bar X_k=(X_{n_k},X_{n_{k+1}})$, and $A_N$ is 
either\footnote{$A_N$ is $1$ only in the coefficients of the polynomial multiplied by $\sig_N^{-1}$, but for the proof to work we actually only need $A_N$ to be bounded.} $0$ or $1$, $c_{k_1,...,k_{3r}}$ are combinatorial coefficients bounded by some constant $C=C_r$, 
and 
\begin{equation}
\label{G-Prod}
G_{k_1,...,k_{\ell_r},N}=\prod_{s=1}^{3r-2}\sig_N^{-j_{k_s}}G_{j_{k_s},k_s}
\end{equation}
for appropriate $0\leq j_{k_s}\leq m_r$ (for some $m_r$ which depends only on $r$), where for $j=0$ we set $G_{0,k}=0$. 
\end{lemma}
\begin{proof}

By Lemma \ref{CondChain}, after conditioning on $\cE=\sig\{X_{n_s}\}=\sig\{X(\cB_N)\}$, the random variables $\textbf{X}_s=\{X_j:\,n_s\leq j\leq n_{s+1}$\} are independent.
Therefore, if we fix a realization $X(\cB_N)=x$ then\footnote{Recall that $\Gamma$ is defined by \eqref{Gamm def}.} 
$$
\Gamma_{t,N,x}(z)=\sum_{k=0}^{N_0}\Gamma_{t,n_k, n_{k+1},x}(z).
$$
Thus, recalling \eqref{DefLambda}, we have
\begin{equation}\label{Lam rep}
(i)^{-s}\sig_N^s\Lambda_{t,N,x}^{(s)}(0)=\sum_{k=0}^{N_0}\Gamma_{t,n_k, n_{k+1},x}^{(s)}(0).
\end{equation}

Next, the coefficient of the polynomials defined on the right hand side of \eqref{Poly} are linear combinations of $A_{t_j,s,N,\cE}, s\leq w_r$. Now, by \eqref{A def}, \eqref{H def}, \eqref{H} and \eqref{H P} each  $A_{t_j,s,N,\cE}$ has the form 
$$
A_{t,s,N,x}=\mathcal P_s \left(d_{t,N,x},u_{t,N,x}, \Lambda_{t,N,x}^{(3)}(0),...,\Lambda_{t,N,x}^{(v_s)}(0)\right)
$$
for some polynomial $\mathcal P_s$ whose degree does not exceed 
$4s-2$, where $v_s$ is some positive integer. Indeed, the term of the smallest order in  the brackets on the left hand side of \eqref{A def} is $d_N$, 
and the largest relevant power of $d_N$ is $3s-2$. On the other hand, the term $H_{t,N,x,r}(t)$ contributes at most $s$ variables among $\Lambda_{t,N,x}^{(u)}(0)$ to $A_{t,s,N,x}$, where there is an actual contribution only if $u\leq s+2$ since $A_{t,s,N,x}$ 
is defined by means of  the coefficients of $H_{t,N,x,r}(h)$, $d_N$ and $u_N$
(see \eqref{A def}).

Overall, we get at most $4s-2$ appearances of variables of the form $(X_{n_k},X_{n_{k+1}})$ which amounts in at most $2(4r-2)=8r-4$ appearances of variables of the form $X_{n_j}$, which is the maximal number of conditioned variables in \eqref{CondB}.
 Now we arrive 
at \eqref{G-Prod} by
taking expectation of the expression in  \eqref{ExForm}, 
 using \eqref{Lam rep} 
and \eqref{d rep},
and the fact
that for every random variable of the form $Q=Q(X_{m_1},...,X_{m_s})$ with $m_\ell\in\cB$ we have 
$$
\bbE\left[\bbE(e^{it_j S_N}|\cE)Q\right]=\bbE\left[Q\cdot\bbE\left[\bbE(e^{it_j S_N}|\cE)|X_{m_1},...,X_{m_s}\right]\right]
$$
$$
=\bbE\left[Q\cdot\bbE(e^{it_j S_N}|X_{m_1},...,X_{m_s})\right].
$$
\end{proof}

Next, by the Markov property we see that $G_{k,s},s\geq 3$ are functions of $(x_{n_k},x_{n_{k+1}})$, where $x=(x_{n_k})_k$.
Arguing like in the proof of Lemma \ref{Pi'' lemma}(iii) we get 
\begin{equation}\label{3 bound}
    \|G_{k,s}\|_{L^\infty}\leq C(1+\sig_{n_k,n_{k+1}}^2), \,s\geq 3
\end{equation}
where the $L^\infty$-norm is with respect to the law of $(X_{n_k},X_{n_{k+1}})$, and the  term $1+\sig_{n_k,n_{k+1}}^2$ appears (instead of just $\sig_{n_k,n_{k+1}}^2$) to cover the case when $\sig_{n_k,n_{k+1}}$ is small (in that case there is no need in working with a block partition and the proof proceeds similarly when working with a single ``block" $(n_k,n_{k+1}))$).  Next, in view of Proposition \ref{d u prop},
for $j=1,2$ we need to estimate $d_N$ and $u_N$ (given by \eqref{u d}) and not only the cumulants. To estimate $d_N$, note that 
\begin{equation}\label{d rep}
\sig_N d_N=i\sum_k(\Gamma_{t,n_k, n_{k+1},x}'(0)-i\bbE[S_{n_k, n_{k+1}}])=\sum_{k=0}^{N_0} G_{k,1}.
\end{equation}
Notice   also that $G_{k,1}$ depend only on $ (x_{n_k},x_{n_{k+1}})$. 
Arguing exactly like in the proof of Lemma~\ref{Pi'' lemma}(i)  we see that, for every $p\in(1,2)$,
\begin{equation}\label{Gk1}
\|G_{k,1}\|_{L^\infty}=O\left((\sig_{n_k,n_{k+1}})^{\frac{2p}{p+2}}(\ln\sig_{n_k,n_{k+1}})^{\frac{2}{p+2}}\right)+O(1)
\end{equation}
where the $O(1)$ term is only needed when $\sig_{n_k,n_{k+1}}=\sqrt{\text{Var}(S_{n_k,n_{k+1}})}$ is small.

To estimate $u_N$, by  Lemma \ref{L1} we see that 
$$
-\sig_N^2u_N=\Gamma_{N}''(0)-\sig_N^2=\Gamma_{t,N,x}''(0)-\text{Var}(S_N|X(\cB_N)=x)+O(\ln^2\sig_N)
$$
uniformly in $x$. Observe now that 
$$
\text{Var}(S_N|X(\cB_N)=x)=\sum_k \text{Var}(S_{n_k,n_{k+1}}|X(\cB_N)=x)
$$
since the blocks between two  consecutive $n_j$'s are conditionally independent 
(see Lemma~\ref{CondChain}). Hence,
\begin{equation}\label{u rep}
-\sig_N^2u_N=\Gamma_{N}''(0)-\sig_N^2=\sum_k(\Gamma_{t,n_k,n_{k+1},s}''(0)-V(S_{n_k,n_{k+1}}|X(\cB_N)=x)) +O(\ln^2\sig_N).
\end{equation}

Next, arguing exactly as in the proof of Lemma \ref{Pi'' lemma}, 
we see that for every $p\in(1,2)$ 
$$
\Gamma_{t,n_k,n_{k+1},x}''(0)=\nonumber \sig_{n_k,n_{k+1}}^2+O\left((\sig_{n_k,n_{k+1}})^{\frac{2p}{p+1}}(\ln \sig_{n_k,n_{k+1}})^{\frac{1}{p+1}}\right)+O(1)
$$
uniformly in $x$, where $\sig_{n_k,n_{k+1}}^2=\text{Var}(S_{n_k,n_{k+1}})$.
We conclude that for every $p\in(1,2)$, 
$$
\|G_{2,k}\|_{L^\infty}=O(1)+O\big((\sig_{n_k,n_{k+1}})^{\frac{2p}{p+1}}(\ln \sig_{n_k,n_{k+1}})^{\frac1{p+1}}\big).
$$

Finally, to prove that the contribution coming from  a nonzero resonant point is negligible it is enough to show that the expressionf in \eqref{Form}  are of order $o(\sig_N^{-(r-1)})$. 
We claim  that 
\begin{equation}\label{G}
\|G\|:=\sum_{k_1,...,k_{\ell_r}}\|G_{k_1,...k_{\ell_r},N}\|_{L^\infty}\leq C
\end{equation}
for some $C$ which depends only on $r$. Note that  \eqref{G} implies that
$$
\left|
\sum_{k_1,...,k_{\ell_r}}\bbE\left[\bbE[e^{it_jS_N}|\bar X_{k_{1}},...,\bar X_{k_{\ell}}]G_{k_1,...,k_{\ell_r},N}\right]
\right|\leq C\sup_{a_1,...,a_{2\ell_r}\in \cB}\left\|
\bbE[e^{it_j S_N}|X_{a_1},\dots, X_{a_{2\ell_r}}]\right\|_{L^1}
$$
and so the first condition is indeed sufficient,  which completes the proof of Proposition \ref{Suff}.

In order to prove \eqref{G}, let us first consider the case
 when one of $j_{k_s}=j$ is larger than $2$. In this case, by \eqref{3 bound}, we have 
$$
\sig_N^{-j}G_{j_{k_s},k_s}=\sig_N^{-j}
\left[O(\sig_{n_{k_s},n_{k_s+1}}^2)+O(1)\right]=
{O\left(\sig_{n_{ k_s},n_{ k_s+1}}^{-1}\right)}
$$
while the other terms are bounded. Thus the contribution to  $\|G\|$ of such terms  is $O(N_0^{\ell_r})\sig_N^{-1}=o(1)$.

 Next we consider the case when  $j_{k_s}$ are either $1$ or $2$ for all $s.$
 If one of them is $1$, then, since the other terms in the product are bounded, we see that the contribution to $\|G\|$ of such terms is dominated by
$$
N_0^{\ell_r}\sum_{k=0}^{N_0}\sig_N^{-1}\|G_{1,k}\|_{L^\infty}.
$$
However,  by \eqref{Gk1}, if $p$ is close enough to $1$ then
$\DS
\|G_{k,1}\|_{L^\infty}=O\left(\sig_N^{3/4}\right)
$
and so 
$$
N_0^{k_\ell}\sum_{k=0}^{N_0}\sig_N^{-1}\|G_{1,k}\|_{L^\infty}=
O\left(N_0^{k_\ell+1}\sig_N^{-1/4}\right)=o(1)
$$
where we have used that $N_0=O(\ln\sig_N)$.

It remains to consider $(k_1,...,k_{\ell_r})$ so that $j_{k_s}=2$ for all $s$. In this case the contribution to $\|G\|$ from such terms is at most
$\DS
N_0^{k_{\ell_r}}\sum_{k=0}^{N_0}\sig_N^{-2}\|G_{2,k}\|_{L^\infty}.
$
Note that by \eqref{Gk2}, if $p$ is close enough to $1$ then
$$
\|G_{2,k}\|_{L^\infty}=O\big(\ln\sig_N (\sig_N)^{3/2}\big)
$$
where we have  
used again that $\sig_{n_k,n_{k+1}}=O(\sig_N)$. Hence 
$$
N_0^{\ell_r}\sum_{k}\sig_N^{-2}\|G_{2,k}\|_{L^\infty}\leq C\sig_N^{-1/2}N_0^{\ell_r+1}\ln\sig_N=o(1).
$$
The proof that the second condition is sufficient for the conditional expansions is similar, we first condition on a finite number of variables, and then 
 apply the above arguments to the 
conditioned chain.

\begin{proof}[Proof of Theorem \ref{Thm Stable Cond}]
The equivalence between $(a)$ and $(b)$ follows by 
Remark \ref{Rem} and Proposition \ref{Suff} 
(note that the case when $M_N(m)\geq  R\ln\sig_N$ was already treated in \S \ref{Sec4}).

The equivalence between $(b)$ and $(c)$ follows from  
the standard 
 fact that
 a sequence of  probability measures on $\{\mu_N\}$ on $\bbZ/m\bbZ$ 
 satisfies $\mu_N(a)=\frac{1}{m}+O(\gamma_N)$ for some sequence
 $\gamma_N$ and all $a\in \bbZ/m\bbZ$ 
 iff  $\hat\mu_N(b)=O(\gamma_N)$ for all $b\in \left(\bbZ/m\bbZ\right)\setminus \{0\}$
where $\hat\mu$ is the Fourier transform of $\mu$
(see e.g. 
\cite[Lemma 6.2]{DH}). 
\end{proof}

\subsection{Prokhorov Theorem for Markov chains.}\label{6.1}
Here we prove Theorem \ref{ThProkhorovMC}. 

\begin{proof}
First, suppose that $M_N\to\infty$ as $N\to\infty.$ We need to show that after
 conditioning on  finitely many $X_j$'s,
the LLT holds. However, the conditioned chain $\tilde X_N$ will also satisfy that 
$\DS \lim_{N\to\infty} \widetilde M_N=\infty.$ Therefore it suffices to show that if
$\DS \lim_{N\to\infty} M_N=\infty$ then the LLT holds. By \cite[Theorem 5.2]{DS}, to this end it is enough to show that $\{f_n\}$ is irreducible, which by \cite[Theorem 4.4]{DS} 
means that  for each resonant point $t$, $\DS \sum_{n=1}^N d_n^2(t)=\infty$ where $d_n(t)$ are
structure constants defined in \cite[\S 2.3.2]{DS}. Namely, taking into account the one step ellipticity,
$$ d_n^2(t)=\iiiint
\left|e^{it \left[f_{n-1}(x_{n-1}')+f_n(x'_{n})-f_{n-1}(x_{n-1}'')-f_n(x''_{n})\right]}-1\right|^2 
d\boldsymbol{P}_n$$
where $d\boldsymbol{P}_n$ stands for 
$\mu_{n-1}(dx_{n-1}') \times \mu_{n-1}(dx_{n-1}'')
\times \mu_n(dx_n')\times \mu_n(dx_n(x''))$.
For resonant $t=\frac{2\pi k}{h}$ the exponential above takes only finitely many values, so
$$ d_n(t)^2\geq c\boldsymbol{P}_n (f_{n-1}(x_{n-1}')+f_n(x'_{n})-f_{n-1}(x_{n-1}'')-f_n(x''_{n})\neq 0 \text{ mod } h) .$$
Note that if $f_n'$ and $f_n''$ are two possible values of $f_n$ with different residues mod $h$ then
for each triplet $x_{n-1}', x_{n-1}'', x_n'$ at most one of 
$$ f_{n-1}(x_{n-1}')+f_n(x'_{n})-f_{n-1}(x_{n-1}'')-f'_n \quad \mathrm{and} \quad 
f_{n-1}(x_{n-1}')+f_n(x'_{n})-f_{n-1}(x_{n-1}'')-f''_n $$
can be zero mod $h.$
Therefore $\DS \sum_{n=1}^N d_n^2(t)\geq c M_N$ and $\{f_n\}$ is irreducible.

Conversely suppose that there exists $h$ such that $\sum_n \bbP(f_n\neq m_n \mod h)<\infty.$
By dropping finitely many terms we may assume that the above sum is less than $\varepsilon>0$. Then for each $N$,\;\;
$\DS \bbP(S_N=\mathfrak{m}_N \mod h)\geq 1-\varepsilon$ where $\mathfrak{m}_N=\sum_{n=1}^N m_n$. Note that the same holds for the conditional sum since the conditioned chain has similar properties.  Therefore, if $\varepsilon$ is small enough and $N 
$ is large enough we see that the distribution of $S_N$ modulo $h$ is not close to being the uniform one, which by Remark \ref{R5} means that the stable LCLT fails.
\end{proof}

\appendix
\section{Proof of Proposition \ref{LamProp}}\label{Appendix A}

First, note that the standing assumption of Proposition \ref{LamProp} can be written as follows: for every $3\leq j\leq r+1$ we have
\begin{equation}\label{bar der}
\sup_{t\in[-\del_r\sig_N,\del_r\sig_N]}\left|\bar\cL_N^{(j)}(t)\right|=O(\sig_N^{-(j-2)}).  
\end{equation}

 Let $\cP_{k,N}$ denote the $k$-th Taylor polynomial of $\bar\cL_N$ around zero, and let $\cR_{k,N}$ denote the corresponding Taylor remainder.

We first need two lemmata.

\begin{lemma}\label{Le2}
Under the assumptions of Proposition \ref{LamProp} we have the following.
Let $3\!\leq\! k\!\leq\! r\!+\!2$. Then there is a constant $B_k>0$ such that 
for every $t\in[-\sig_N,\sig_N]$ we have
$$|\cP_{k,N}(t)|\leq B_k\sig_N^{-1}|t|^3=B_kt^2|t/\sig_N|.$$
Therefore,  for every $t\in[-\del_k\sig_N,\del_k\sig_N]$, $\del_k=\frac1{4B_k}$, we have
$$|e^{\cP_{k,N}(t)}|\leq e^{t^2/4}.$$
\end{lemma}

\begin{lemma}\label{Le1}
Under the assumptions of Proposition \ref{LamProp},
there exists a constant $\del_0>0$ such that 
  for every  
 $t$ with $|t/\sig_N|\leq \del_0$ we have $|e^{\bar\cL_N(t)}|\leq e^{t^2/4}$.
\end{lemma}

\begin{proof}[Proof of Lemmas \ref{Le2} and \ref{Le1}]
Let us first prove Lemma \ref{Le2}.
Since $\bar\cL_N^{(k)}(0)=0$ for $k=0,1,2,3$ we have 
$$
\cP_{k,N}(t)= \sum_{j=3}^k\frac{\bar\cL_N^{(j)}(0)}{j!}t^j
$$
and so by \eqref{bar der} (applied with $t=0$), if
 $|t/\sig_N|\leq 1$ then for some constants $A_k,B_k>0$,
$$
|\cP_{k,N}(t)|\leq \sum_{j=3}^k\frac{|\bar\cL_N^{(j)}(0)|}{j!}|t|^j
\leq A_kt^2\sum_{j=3}^k|t/\sig_N|^{j-2}\leq A_k t^2\sum_{j=3}^k |t/\sig_N|\leq B_k|t|^3\sig_N^{-1}.
$$
Hence, if $\DS \left|\frac{t}{\sig_N}\right|\!\leq\!\frac{1}{4B_k}\!=\!\del_k$ then $\DS |\cP_{k,N}(t)|\!\leq\!\! \frac{t^2}{4}$ and so 
$
\DS \left|e^{\cP_{k,N}(t)}\right|\!\leq\! e^{t^2/4}
$
 proving Lemma \ref{Le2}.

Next, since  $\cP_{2,N}(t)=0$, we have $\Lambda_N(t)=\cP_{2,N}(t)+\cR_{2,N}(t)=\cR_{2,N}(t)$. Then, by 
 the Lagrange form of the Taylor remainder, we can write $\cR_{2,N}(t)=\frac{t^3\bar\cL_N'''(t_1)}{3!}$ for some $t_1$ such that $|t_1|\leq|t|$. Therefore, 
by  \eqref{bar der},
$$
|\cR_{2,N}(t)|\leq C_3|t|^3\sig_n^{-1}=C_3t^2\cdot|t/\sig_N|,\,  \text{if }\, |t/\sig_N|\leq\del_3.
$$
Thus when also $|t/\sig_N|<\frac{1}{4C_3}$   we have $|\bar\cL_N(t)|=|\cR_{2,N}(t)|<t^2/4$, and Lemma \ref{Le1}  follows.
\end{proof}


\begin{proof}[Proof of Proposition \ref{LamProp}]
Using the Lagrange form of Taylor remainders together with \eqref{bar der} we get that if $|t/\sig_N|\leq \del_{r+1}$ then
\begin{equation}\label{Zero}
|\cR_{r+2,N}(t)| \leq C_r|t|^{r+2}\sig_N^{-r-1}.
\end{equation}
Next, by the mean value theorem and Lemmata \ref{Le2} and \ref{Le1} there are constants $\del_r'>0$ and $C_0>0$  so that if $|t/\sig_n|\leq\del_r'$ then
\begin{equation}\label{One.1}
\left|e^{\bar\cL_N(t)}-e^{\cP_{r+2,N}(t)}\right|\leq C_0e^{t^2/4}|\cR_{r+2,N}(t)|.
\end{equation}
Moreover, by Lemma \ref{Le2} and the Lagrange form of Taylor remainders (applied with the function $e^x$ around the origin),
\begin{equation}\label{Two}
\left|e^{\cP_{r+2,N}(t)}-
\left(1+\sum_{j=1}^{r}\frac{\left(\cP_{{r+2},N}(t)\right)^j}{j!}\right)\right|\leq 
D_{r}e^{t^2/4}\sig_N^{-r-1}|t|^{3(r+2)}
\end{equation}
where $D_{r}>0$ is some constant. 
Combining \eqref{Zero}, \eqref{One.1} and \eqref{Two}, we see that there exists a constant $\del_r''>0$ such that for every real $t$ with $|t/\sig_N|\leq \del_{r}''$ we have 
$$
\left|e^{\bar\cL_N(t)}-\left(1+\sum_{j=1}^{r}\frac{\left(\cP_{{r+2},N}(t)\right)^j}{j!}\right)\right|
\leq Ce^{t^2/4}\sig_N^{-(r+1)}\max\left(|t|^{r+2},|t|^{3(r+2)}\right).
$$
Next, in order to get the specific form  $\cH_{N,r}(t)$ of the approximating polynomials (like in Proposition \ref{LamProp}), after expanding $\left(\cP_{r+2,N}(t)\right)^j$ in powers of $t$, we disregard all the coefficient that are of order $O(\sig_N^{-r-1})$. 
\end{proof}

\section{Proof of Theorem \ref{RPF}}\label{App B}
Theorem \ref{RPF} can be proven using complex cones contractions similarly to \cite[Ch. 6]{HK}. However, this requires a long presentation, and instead we will prove the theorem using the simpler approach in \cite[Appendix D]{DH New BE}.

First, as in the proof of  Lemma \ref{FiberQuantLemma} 
we can extend the Markov chain $(X_ j)_{j\geq0}$ to a two sided uniformly elliptic chain $(X_j)_{j\in\bbZ}$. By setting $U_j=Y_j=0$ for $j\leq 0$ we see that we can also define $R_z^{(j)}$ for negative indexes.

Next, let
$$
\cB_1=\{u=(u_j)_{j\in\bbZ}: u_j\in L^2(\mu_j), \sup_j\|u_j\|_{L^2(\mu_j)}<\infty\}.
$$
Then $\cB_1$ is a Banach space with respect to the norm 
$$
\|u\|=\sup_j\|u_j\|_{L^2(\mu_j)}.
$$
Next let us define $\cB=\cB_1\oplus \bbC$. 

Now, for each $(u,z)\in\cB$ and $j\in\bbZ$ we define the operator $\cA_{j}^{(u,z)}:B_{j+1}\to B_j$ by
$$
\cA_j^{(u,z)}g(x)=\bbE[e^{iu_{j+1}(X_{j+1}))+zY_{j+1}}g(X_{j+1})|X_j=x].
$$
Then $R_z^{(j)}=\cA_j^{(z,u)}$ for the appropriate choice of functions $u_j$.
\begin{lemma}\label{An Lem}
The map $(u,z)\to \cA_j^{(u,z)}$ is $C^1$ and 
$$
\sup_{j, \|u\|\leq 1, |z|\leq 1}(\|\cA_j^{(u,z)}\|_\infty+\|D(\cA_j^{(u,z)})\|_\infty)<\infty.
$$
\end{lemma}
\begin{proof}
First, the uniform boundedness of $\|\cA_j^{(u,z)}\|_\infty$ is clear since $Y_j$ is a uniformly bounded sequences.

Next, since $\|Y_j\|_\infty\leq K<\infty$ the maps $(u,z)\to e^{zY_{j+1}}\in B_{j+1}$ are analytic and are uniformly bounded.
Moreover, $\cA_j^{(u,z)}$ is the composition of $\cA_j^{(u,0)}$ and the map $(u,z)\to \cP_j(u,z)$, where $\cP_j(u,z)$ is the multiplication operator given by $\cP_j(u,z)[g]=ge^{zY_{j+1}}$.
Thus, it is enough to prove the continuous differentiability and the uniform boundedness of the deriviatives when freezing $z=0$ and differentiating in the $u$ direction.
In the case $z=0$
the lemma follows from the following estimates 
$$
\left|\bbE[g(X_{j+1})(e^{i(u_{j+1}(X_{j+1})+v(X_{j+1}))}-e^{iu_{j+1}(X_{j+1})}(1+iv(X_{j+1}))|X_j=x]\right|
$$
$$
\leq \|g\|_\infty\bbE[|e^{iv_{j+1}(X_{j+1})}-1-iv_{j+1}(X_{j+1})|X_j=x]\leq 
C\|g\|_\infty\bbE[|e^{iv_{j+1}(X_{j+1})}-1-iv_{j+1}(X_{j+1})|]
$$
$$
\leq
C\|g\|_\infty\bbE[v^2(X_{j+1})]=C\|g\|_\infty\|v_{j+1}\|_{L^2}^2
$$
where in the second inequality we have used that  for all real $v$ we have
$
|e^{iv}-1-iv|\leq |v|^2.
$
Thus, 
$$
[D_u(\cA_j^{(u,0)})v_{j+1}](g)=\cA_j^{(u,0)}(v_{j+1}g)
$$
which is continuous in $u$ since for every $q\in \cB_1$, 

\hskip1cm $\DS
\bbE_{\mu_{j+1}}[|v_{j+1}|\cdot|e^{iq_{j+1}}-1|]\leq \|v_{j+1}q_{j+1}\|_{L^1}\leq \|v_{j+1}\|_{L^2}\|q_{j+1}\|_{L^2}\leq \|v\|\|q\|.
$
\end{proof}

\subsection{Proof of Theorem \ref{RPF}}
First, as a particular case of \cite[Theorem 14]{DH BE} we see that there are triplets like in Theorem \ref{RPF} corresponding to the choice $u_j=0$ and $z=0$. Relying on Lemma \ref{An Lem} we apply \cite[Theorem D.2]{DH New BE} to get triplets like in Theorem \ref{RPF} for all  $u$ and $z$ with $\|u\|+|z|\leq r_0$, where $r_0>0$ is a constant. It also follows that these triplets are $C^1$ functions of $(u,z)$ when $z$ is viewed as a vector in $\bbR^2$. Note that we also took into account that $\la_j(0)=1$ when $u=0$ and so when $\|u\|+|z|$ is small enough then $(1-a)^n\leq |\la_{j,n}(z)|\leq (1+a)^n$ for $a=a(\|u\|+|z|)$ such that $\DS \lim_{\|u\|+|z|\to 0}a(\|u\|+|z|)=0$. Thus \eqref{Exp Conv final.0.1.1} follows from the corresponding estimate in \cite[Theorem D.2]{DH New BE} after dividing by $\la_{j,n}(z)$. 

To prove analyticity in $z$ we note that the operators $R_{z}^{(j)}$ are analytic in $z$ when $z$ is viewed as a complex number. Next, we show that for each fixed $u$ we can express $\la_j(z)$, $h_j^{(z)}$ and $\nu_j^{(z)}$ as uniform limits of the iterates of the operators $\cA_j^{(u,z)}$, which is enough to deduce analyticity of $\la_j(z)$, $h_j^{(z)}$ and $\nu_j^{(z)}$. 

To that end we first notice that when $\|u\|+|z|$ is small then $h_j^{(z)}$ are uniformly close to $1$. Thus, by reducing $\|u\|+|z|$ is necessary we can always assume that $|h_j^{(z)}(x)|\geq \frac12$ for all $x$. Hence by 
 \eqref{Exp Conv final.0.1.1},
$$
\left\|\nu_j^{(z)}(\cdot)-\frac{R_z^{j-n,n}(\cdot)}{\la_{j-n,n}(z)h_{j-n}^{(z)}}\right\|_\infty\leq 2C\delta^n.
$$
Therefore, uniformly in $z$ we have
$$
\nu_{j}^{(z)}=\frac{\nu_{j}^{(z)}}{\nu_{j}^{(z)}(\textbf{1})}=\lim_{n\to\infty}\frac{R_z^{j-n,n}}{R_z^{j-n,n}\textbf{1}}
$$
and so it is a uniform limit of analytic functions. Here we identified between $\nu_{j}^{(z)}$ and the operator $g\to \nu_{j}^{(z)}(g)\textbf{1}$.
Observe next that 
$\la_j(z)=\nu_j(R_j^{(z)}\textbf{1})$. Thus, also $\la_j(z)$ is analytic in $z$. Finally, taking $q=1$ in \eqref{Exp Conv final.0.1.1} we see that, uniformly in some complex neighborhood of $0$ we have
$\DS
h_j^{(z)}=\lim_{n\to\infty}\frac{R_{z}^{j-n,n}\textbf{1}}{\la_{j-n,n}(z)}.
$
Thus, also $h_j^{(z)}$ are analytic in $z$.

\end{document}